\documentclass[r]{siamart171218}
\usepackage[english]{babel}
\usepackage{amsmath}
\usepackage{amssymb}
\usepackage{amsfonts}
\usepackage{array}
\usepackage{graphicx}
\usepackage{epsfig}
\usepackage{float}
\usepackage{color}
\usepackage{enumitem}  
\usepackage{epstopdf}
\usepackage{stmaryrd}
\usepackage{standalone}
\usepackage{tikz, subfigure}
\usepackage{yfonts}
\usepackage{bm}
\usepackage{cancel,comment}

\usetikzlibrary{shapes, patterns, arrows}
\usetikzlibrary{calc}

\usepackage{booktabs}  
\usepackage{capt-of}
\usepackage{multirow}
\usepackage[mathlines]{lineno}

\definecolor{lightblue}{rgb}{0.22,0.45,0.70}

\newcommand\bb{\boldsymbol b}

\newcommand\bq{\boldsymbol q}

\newcommand\bu{\boldsymbol u}
\newcommand\bv{\boldsymbol v}
\newcommand\bzeta{\boldsymbol \zeta}
\newcommand\cero{\boldsymbol 0}
\newcommand\bnabla{\boldsymbol \nabla}
\newcommand\beps{\boldsymbol \varepsilon}
\newcommand\bnu{\boldsymbol \nu}
\newcommand\bsigma{\boldsymbol \sigma}

\newcommand{\enorm}[1]{{\left\vert\kern-0.25ex\left\vert\kern-0.25ex\left\vert #1 
    \right\vert\kern-0.25ex\right\vert\kern-0.25ex\right\vert}}

\renewcommand{\AA}{\mathcal{A}}
\newcommand{\BB}{\mathcal{B}}

\newcommand{\cT}{\mathcal{T}}
\newcommand{\Ker}{\operatorname{Ker}}

\newcommand{\vdiv}{\operatorname{div}}

\newcommand{\bI}{\mathbf{I}}
\newcommand{\bL}{\mathbf{L}}
\newcommand{\bH}{\mathbf{H}}

\newcommand{\jump}[1]{\ensuremath{[\![#1]\!]} }
\newcommand{\avg}[1]{\ensuremath{\left\{\!\!\left\{#1\right\}\!\!\right\}} }

\newsiamthm{remark}{Remark}

\headers{Robust preconditioners for perturbed saddle-point problems}{W.M. Boon, M. Kuchta, K.-A. Mardal, R. Ruiz-Baier}
\title{Robust preconditioners
for perturbed saddle-point problems 
and conservative discretizations of Biot's equations utilizing total pressure
\thanks{Submitted to the editors DATE.
\funding{
	WMB acknowledges support from the Dahlquist Research Fellowship, funded by Comsol AB. MK acknowledges support from the Research Council of Norway (NFR) grant no. 280709.
	KAM acknowledges support from the Research Council of Norway, grant 300305 and 
	301013. RRB acknowledges support from the Monash Mathematics Research Fund S05802-3951284.
}}}
\author{Wietse M. Boon\thanks{Department of Mathematics,
KTH Royal Institute of Technology,
Lindstedtsv\"agen 25, 114 28, Stockholm,
Sweden, \email{wietse@kth.se}.} 
\and Miroslav Kuchta\thanks{Simula Research Laboratory, 1325 Lysaker, Norway, 
\email{miroslav@simula.no}.} 
\and Kent-Andre Mardal\thanks{Department of Mathematics, Division of Mechanics, University of Oslo, Norway; and 
 Simula Research Laboratory, 1325 Lysaker, Norway, 
\email{kent-and@math.uio.no}.} 
\and 
Ricardo Ruiz-Baier\thanks{School of Mathematics, 
Monash University, 
9 Rainforest Walk, 
Clayton 3800 VIC,  Australia,  
\email{ricardo.ruizbaier@monash.edu}.}}
\date{\today}

\begin{document}

\maketitle

\begin{abstract}
We develop robust solvers for a class of perturbed saddle-point problems arising in the study of a second-order elliptic equation in mixed form (in terms of flux and potential), and of the four-field formulation of Biot's consolidation problem for linear  poroelasticity (using displacement, filtration flux, total pressure and fluid pressure). The stability of the continuous variational mixed problems, which hinges upon using adequately weighted spaces, is addressed in detail; and the efficacy of the proposed preconditioners, as well as their robustness with respect to relevant  material properties, is demonstrated through several numerical experiments. 
\end{abstract}

\section{Introduction}

Disparity of model parameters is a phenomenon commonly encountered in a variety of applications, and it is of paramount importance that the formulation of multiphysics problems and the design of discretizations and efficient solvers is robust with respect to at least some of the parameters with wide value ranges. We will here consider the equations of linear poroelasticity, where fluid flows in isothermal deformable porous media, assuming that the solid skeleton undergoes small strains. 
The poroelastic Biot equations that form the subject of this study are: 
\begin{subequations}\label{biot0}
\begin{align}
\label{biot1}
	- \bnabla \cdot( 2 \mu \beps(\bu) + (\lambda \nabla \cdot \bu - \alpha p) \bI) &= \bb & \text{in $\Omega\times(0,t_{\text{end}}]$},\\
\label{biot2}
	K^{-1} \bq + \nabla p &= \cero & \text{in $\Omega\times(0,t_{\text{end}}]$},\\
\label{biot3}
	\partial_t(c p + \alpha \nabla \cdot \bu) + \nabla \cdot \bq &= f& \text{in $\Omega\times(0,t_{\text{end}}]$},
\end{align}
\end{subequations}
equipped with suitable boundary (and initial) conditions to be specified later. 
Here, $\bu(t):\Omega\to\mathbb{R}^d$ is the solid displacement vector, $\bq(t):\Omega\to\mathbb{R}^d$ is the Darcy flux or percolation velocity, $p(t):\Omega\to\mathbb{R}$ is the fluid pressure,
the symbol $\partial_t$ denotes the partial derivative with respect to time, $\bb(t):\Omega\to\mathbb{R}^d$ is a 
prescribed body force per unit volume of the porous medium, the symmetric part of the displacement (row-wise) gradient defines the infinitesimal strain tensor  $\beps(\bu) = \frac{1}{2}(\nabla\bu+\nabla\bu^{\tt t})$, the parameters $\lambda,\mu$ are the Lam\'e constants of the solid, $K$ is the hydraulic conductivity (ratio between the material permeability and fluid viscosity), $f(t):\Omega\to\mathbb{R}$ is a  
source or sink of fluid mass, and $c,\alpha$ are the total storage capacity and Biot-Willis poroelastic coefficient, respectively. 

Several types of discretizations for \eqref{biot0} are available from the literature, including mixed and continuous elements, least-squares mixed, stabilized $H$(div)-conforming and other non-conforming schemes, adaptive mixed methods, weak Galerkin, enriched Lagrangian, and hybrid finite-volume finite element methods (see, e.g.,  
\cite{phillips07,berger15, zeng19,li20,rl17,wheeler14,hu2017nonconforming,yi13,sun17,harper20}
and the references therein). 

A main challenge for these equations is the construction of solvers that scale properly for 
nearly incompressible solids where the Lam\'e dilation modulus tends to infinity, as well as in the case of nearly incompressible fluids, for which the constrained specific storage coefficient approaches zero, or the nearly impermeable regime where the hydraulic conductivity is very small. These scenarios entail not only a complication at the practical and implementation level, but also a difficulty inherent to the functional setting of the abstract formulation (see, e.g., \cite{piersanti2020parameter,hong18,yi17}). 
In more detail, for almost incompressible solids ($\lambda\gg\mu$), the primal form of the elasticity equation, used in \eqref{biot1}, here scaled by $\lambda$, 
\[
- \bnabla \cdot( 2 \frac{\mu}{\lambda} \beps(\bu) + (\nabla \cdot \bu) \bI) = \bb \  \text{in} \  \Omega,\\
\]
is known to suffer from locking when using standard elements such as
Lagrange elements. The reason is that the problem is a singular perturbation problem, where stability in $\bH^1(\Omega)$ decays as $\mu/\lambda$ tends to zero and where stability can 
only be obtained in $\bH(\mbox{div}, \Omega)$. 
A remedy is to use elements that are stable in both $\bH(\mbox{div}, \Omega)$ and $\bH^1(\Omega)$ such as in \cite{hong18,hong2019conservative} where stabilized Brezzi-Douglas-Marini (BDM) elements are employed. 
Another alternative is to employ a technique similar to  Herrmann's method~\cite{herrmann1965elasticity} where an additional solid pressure, 
$p=\lambda\nabla\cdot\bu$, is introduced. 
It has been shown that a straightforward application of Herrmann's method is unstable, but that the technique can be adjusted such that the discretization becomes stable~\cite{lee17,oyarzua16}
for displacement-pressure formulations, the so-called total-pressure formulation. 
The method was extended to conservative formulations, i.e. displacement-flux-pressure, in~\cite{kumar20}, but robustness with
respect to all parameters was not established. 

A second singular perturbation problem occurs when the hydraulic conductivity ($K$) tends to zero. 
To prevent
non-physical pressure oscillations, mixed schemes involving both flux and pressure are often used, i.e., 
ignoring for the moment the elastic deformation, the equations
read: 
\begin{align*}
	K^{-1} \bq + \nabla p &= \cero & \text{in $\Omega\times(0,t_{\text{end}}]$},\\
	\partial_t c p  + \nabla \cdot \bq &= f& \text{in $\Omega\times(0,t_{\text{end}}]$}.
\end{align*}
Hence, upon time-discretization, this system is a mixed Darcy problem with a 
lower order perturbation term for the pressure and we 
will consider the cases where one or both of $K\rightarrow 0$ and $c\rightarrow0$
is allowed. It is seen that if the perturbation is sufficiently small (such that the additional term is bounded by the pressure norm), then the perturbed problem is well-posed if a weighted $L^2$-norm, i.e., $K^{\frac12} L^2(\Omega)$, is used 
for the pressure. This observation has been frequently
employed~\cite{hong18, powell2003optimal, vassilevski2013block, vassilevski2014mixed} in various porous media flow applications. 
However, if the hydraulic conductivity ratio $K$ is small, then 
(the fluid part of) the perturbation cannot be bounded by the $K^{\frac12} L^2(\Omega)$-norm, preventing a robust stability result. With this in mind, a convenient rescaling of the employed norms seems to produce better results, as recently suggested in \cite{baerland2018observation}. 

We also mention that for conservative
Biot formulations such as \eqref{biot1}--\eqref{biot3}, the
stability, i.e. the inf-sup condition,  of the porous media problem can be weakened, as observed in~\cite{hu2017nonconforming, mardal2020accurate}. 
We will show, for the total pressure formulation
of the conservative form of Biot's equations \eqref{biot1}--\eqref{biot3}, that the scaling of the fluid pressure cannot be chosen independently of
the coupling to the solid displacement and that
the stability of the fluid pressure in $L^2(\Omega) + K^{\frac12} H^1(\Omega)$ is crucial. 

A key tool for our stability analysis is the seminal 
paper~\cite{braess1996stability} which analyzed saddle-point problems with penalty terms corresponding to singular perturbation problems. Therein, it is shown that 
depending on the penalty term, the perturbation may either stabilize or
de-stabilize the saddle-point problem. The Biot equations
in study here involve two saddle-point problems with 
penalties corresponding to two singular perturbation problems
that may be strongly coupled. The analysis leads us to utilize 
non-standard Sobolev spaces to untangle the precise stability 
problems required in both the continuous and discrete settings.

The paper is structured as follows. 
The motivating problem of Biot consolidation and its variational formulation are presented in the remainder of this section. Then in Section~\ref{sec:abstract} we give an overview of the analysis of perturbed saddle-point problems following \cite{braess1996stability}. In Section~\ref{sec:helmholtz}, this theoretical framework is used to show that a generalized Poisson equation in mixed form with Dirichlet boundary conditions is stable in appropriately weighted norms, and there we also discuss the case of Neumann boundary conditions. Section~\ref{sec:biot} contains an application of the theory to the four-field formulation of Biot equations. In Section~\ref{sec:discrete} we make precise the norms and spaces required at the discrete level, and in Section~\ref{sec:results} we collect numerical results that test the performance of the proposed block preconditioners for the modified Poisson equation and the Biot consolidation system.

\subsection{Problem formulation}
Let us consider the time domain $t\in(0,t_{\mathrm{end}}]$ and an open, bounded connected Lipschitz spatial domain $\Omega\subset \mathbb{R}^d$, $d=2,3$ on which 
 the Biot equations in quasi-static form, \eqref{biot0},  are posed.

We introduce the total pressure (the sum of the volumetric contributions to the poroelastic Cauchy stress, cf. \cite{lee17,oyarzua16}) as
\begin{align}
	p_T &= \lambda \nabla \cdot \bu - \alpha p. \label{eq: def p0}
\end{align}
We substitute \eqref{eq: def p0} in the momentum balance equation and use it
to rewrite the volumetric term in the mass balance equation:
\begin{align*}
	\alpha \nabla \cdot \bu = \frac{\alpha}{\lambda} p_T + \frac{\alpha^2}{\lambda} p.
\end{align*}

This leads to the four-field formulation of Biot's equation (see, e.g., \cite{hong18,kumar20}) written in operator form 
\begin{align*}
	\begin{bmatrix}
		- \bnabla \cdot (2\mu \beps) & & - \nabla & \\
		& K^{-1} & & \nabla \\
		-\nabla \cdot & & \frac{1}{\lambda}  & \frac{\alpha}{\lambda} \\
		& \nabla \cdot & \frac{\alpha}{\lambda} \partial_t & (c + \frac{\alpha^2}{\lambda}) \partial_t
	\end{bmatrix}
	\begin{bmatrix}
		\bu \\ \bq \\ p_T \\ p
	\end{bmatrix}
	=
	\begin{bmatrix}
		\bb \\ \cero \\ 0 \\ f
	\end{bmatrix}.
\end{align*}

Regarding boundary conditions, we assume that the boundary $\partial\Omega = \Gamma^{\bu}\cup \Gamma^{\bsigma}$ with $\Gamma^{\bu}\cap \Gamma^{\bsigma} = \emptyset$ and $|\Gamma^{\bu}|\neq 0 \neq |\Gamma^{\bsigma}|$, splits in two sub-regions: 
$\Gamma^{\bu}$ where displacement and normal filtration flux are prescribed (the solid is clamped and the fluid slips), and $\Gamma^{\bsigma}$ where we set zero total traction and 
zero fluid pressure
\begin{subequations} \label{eqs: bdry and init conditions}
\begin{align}\label{eq:bc1}
\bu  = \cero \quad \text{and} \quad \bnu\cdot\bq & = 0 & \text{ on $\Gamma^{\bu}\times(0,t_{\text{end}}]$},\\
\label{eq:bc2}
[2\mu\beps(\bu)-p_T\bI]\bnu  = \cero \quad \text{and} \quad 
p & = 0 & \text{ on $\Gamma^{\bsigma}\times(0,t_{\text{end}}]$},
\end{align}
where $\bnu$ is the unit normal 
vector on the boundary $\partial\Omega$. 
We also suppose that the system is initially at rest 
\begin{equation}\label{eq:initial}
\bu = \boldsymbol{0}, \quad p=0, \quad \text{in $\Omega\times\{0\}$.}
\end{equation}
\end{subequations}

In the time-discrete setting, let $\tau$ be the time-step, 
let $\bq_\tau := \tau \bq$,  and 
group the displacement and flux unknowns into a vector $\vec{\bu}$, and the total pressure and fluid pressure into $\vec{p}$ so
that the vector of unknowns (at the current time step)
is $(\vec\bu,\vec{p})^{\tt t} = (\bu,\bq_\tau,p_T,p)^{\tt t}$.  
After a rescaling of the equations similar to \cite{hong18}, we have the operator:
\begin{align}\label{eq:biot}
	\AA \begin{bmatrix}\vec\bu\\ \vec{p}\end{bmatrix} &:= 
	\begin{bmatrix}
		- \bnabla \cdot (2\mu \beps) & & - \nabla & \\
		& (\tau K)^{-1} & & \nabla \\
		-\nabla \cdot & & \frac{1}{\lambda}  & \frac{\alpha}{\lambda} \\
		& \nabla \cdot & \frac{\alpha}{\lambda} & c + \frac{\alpha^2}{\lambda}
	\end{bmatrix}
	\begin{bmatrix}
		\bu \\ \bq_\tau \\ p_T \\ p
	\end{bmatrix}.
\end{align}

Note that from the time-discrete formulation \eqref{eq:biot} and from the setup of boundary conditions \eqref{eq:bc1}-\eqref{eq:bc2}, the natural trial and test spaces (before scaling) for displacement, filtration flux, total pressure, and fluid pressure, are respectively 
$$ \bH_{\Gamma^{\bu}}^1(\Omega), \quad 
\bH_{\Gamma^{\bu}}(\vdiv,\Omega), \quad 
L^2(\Omega), \quad L^2(\Omega).$$

Note also that system \eqref{eq:biot} adopts the structure:
\begin{align}\label{eq:ABBC}
	\begin{bmatrix}
		A & -B^{\tt t} \\
		B & C
	\end{bmatrix}
	\begin{bmatrix}
		\vec\bu \\ \vec{p}
	\end{bmatrix}
	=
	\begin{bmatrix}
		f \\ g
	\end{bmatrix},
\end{align}
with $A$ and $C$ symmetric, positive semi-definite operators, and where the right-hand side vectors $f,g$ contain contributions from the body load and volumetric source, as well as from quantities in the previous time step that arise from the discretization in time. More precisely, we have the weak formulation: Find $(\vec{\bu},\vec p)\in \bigl(\bH_{\Gamma^{\bu}}^1(\Omega)\times 
\bH_{\Gamma^{\bu}}(\vdiv,\Omega)\bigr)\times  
\bigl(L^2(\Omega) \times L^2(\Omega)\bigr)$ such that 
\begin{align*}
    a(\vec{\bu},\vec{\bv})  +b(\vec{\bv},\vec{p}) & = F(\vec{\bv}) & 
    \forall \vec{\bv} &\in \bH_{\Gamma^{\bu}}^1(\Omega)\times 
\bH_{\Gamma^{\bu}}(\vdiv,\Omega),\\
    b(\vec{\bu},\vec{q}) - c(\vec{p},\vec{q})& = -G(\vec{q}) & 
    \forall \vec{q} &\in L^2(\Omega) \times L^2(\Omega),
\end{align*}
where $\vec{\bv} = (\bv, \bzeta,q_T,q)$, and the bilinear forms and functionals adopt the form 
\begin{align*}
a(\vec{\bu},\vec{\bv})& := 2\mu \int_\Omega \beps(\bu):\beps(\bv) +\frac{1}{\tau K} \int_\Omega \bq_\tau \cdot \bzeta,\\
b(\vec{\bv},\vec{q}) &:= \int_\Omega \nabla\cdot\bv\, q_T - \int_\Omega \nabla\cdot\bzeta\, q,\\
c(\vec{p},\vec{q}) &: = \frac{1}{\lambda}\int_\Omega p_T q_T + \frac{\alpha}{\lambda}\int_\Omega p\,q_T + \frac{\alpha}{\lambda}\int_\Omega p_T\,q +  \bigl[c+\frac{\alpha^2}{\lambda}\bigr]\int_\Omega p\,q,\\
G(\vec{q}) &:= \int_\Omega\biggl( \tau f + \bigl[c+\frac{\alpha^2}{\lambda}\bigr] p^n + \frac{\alpha}{\lambda}p_T^n\biggr) q, \qquad 
F(\vec{\bv}) : = \int_\Omega \bb\cdot\bv,
\end{align*}
where $p^n,p_T^n$ denote the approximations of fluid and total pressure on the previous iteration of backward Euler's method. 

Let us point out that using $(\vec\bu,\vec{p})^{\tt t}$ as a test function, we are led to the following \emph{poroelastic energy norm}:
\begin{align} \label{eq: generated semi-norms}
	\langle \AA (\vec\bu,\vec{p}), (\vec\bu,\vec{p}) \rangle
	= 2 \mu \| \beps(\bu) \|_{0,\Omega}^2 
	+ (\tau K)^{-1} \| \bq_\tau \|^2_{0,\Omega} 
	+ \frac{1}{\lambda} \| p_T + \alpha p \|_{0,\Omega}^2 
	+ c \| p \|_{0,\Omega}^2.
\end{align}
However, an issue with writing a global 
multilinear form and trying to analyze its 
stability is that this naturally induced  
semi-norm does not take into account the term $\nabla \cdot \bq$ and therefore one loses separate control over $p_T$ and $p$ if $c = 0$. In particular, the operators $A$ and $C$ do not possess sufficient coercivity to ensure this (see, e.g, \cite{boffi2013mixed,Gatica2014}). We thus require a more involved strategy in order to obtain a stability bound in a stronger norm than \eqref{eq: generated semi-norms}. This will be presented in Section~\ref{sec:biot}, for which we first need to discuss theoretical aspects of perturbed saddle-point systems, which we exemplify with a simpler problem. 

We also point out that the parabolic-elliptic nature of the coupled system may suggest, as an alternative to the monolithic approach leading to \eqref{eq:biot}, to use operator splitting techniques that allow to solve smaller and better conditioned systems in an iterative manner, as studied in, e.g., \cite{ahmed20,both17,mikelic14}; however we do not address those lines here.

\section{Abstract analysis of perturbed saddle-point problems}\label{sec:abstract}
Typically, the stability analysis of perturbed saddle-point problems of type \eqref{eq:ABBC}, posed on $V \times Q$, assumes a given norm on the space $Q$ and uses this norm in its assumptions on the bilinear forms. However, we obtain two different types of control of the solution $p \in Q$, through the operators $B$ and $C$, respectively. It is essential in the context of robust preconditioning to understand these two effects so that the dependencies on model parameters can be properly captured. This section therefore presents an analysis of \eqref{eq:ABBC} with the use of two different (semi-)norms, reflecting the roles that $B$ and $C$ play. 
For this, we rely on the analysis presented in \cite{braess1996stability}.

We start by introducing notation. Let $V$ and $Q_b$ be two Hilbert spaces endowed with norms $\| \cdot \|_V$ and $\| \cdot \|_b$ that are possibly parameter-dependent. Let $Q$ be a dense linear subspace of $Q_b$. We have three bilinear forms $a: V \times V \to \mathbb{R}$, $b: V \times Q_b \to \mathbb{R}$, and $c: Q \times Q \to \mathbb{R}$, of which we assume that $a$ and $b$ are continuous and that $a$ and $c$ are symmetric and positive (semi-)definite, i.e.
\begin{align*}
	a(u, v) &\lesssim \| u \|_V \| v \|_V, & 
	b(u, q) &\lesssim \| u \|_V \| q \|_b, & \forall u, v &\in V, \ \forall q \in Q \\
 	a(u, v) &= a(v, u), & a(v, v) &\ge 0, & \forall u, v &\in V, \\
 	c(p, q) &= c(q, p), & c(q, q) &\ge 0, & \forall p, q &\in Q.
\end{align*} 
Here, we use the notation $x \lesssim y$ to denote that a constant $c_0 > 0$ exists, independent of model parameters such that $x \le c_0 y$. The relation $\gtrsim$ has analogous meaning and we denote $x \eqsim y$ if $x \lesssim y \lesssim x$.

Let $c$ generate the (semi-)norm
\[
	| p |_c^2 := c(p, p), \qquad \forall p \in Q,
\]
and we assume that $Q$ is a complete space endowed with the norm 
$\| p \|_Q^2 := \| p \|_b^2 + | p |_c^2$.

The linear operators associated to $a(\cdot, \cdot)$, $b(\cdot, \cdot)$, and $c(\cdot, \cdot)$ are denoted by $A: V \to V'$, $B: V \to Q_b'$, and $C: Q \to Q'$, respectively. Letting $0 \le t \le 1$ be a scaling parameter, we consider the following problem: \\
Find $(u, p) \in V \times Q_b$ such that
\begin{align}\label{eq:ABBtC}
	\begin{bmatrix}
		A & -B^{\tt t} \\
		B & t^2C
	\end{bmatrix}
	\begin{bmatrix}
		u \\ p
	\end{bmatrix}
	=
	\begin{bmatrix}
		f \\ g
	\end{bmatrix},
\end{align}

We assume that the following bounds, known as the Brezzi conditions, hold:
\begin{subequations} \label{eqs: Brezzi}
	\begin{align}
		\forall v &\in \Ker B, & 
		a(v, v) &\gtrsim \| v \|_V^2,  \label{eq: coercivity A}\\
		\forall p &\in Q_b, &
		\sup_{v \in V} \frac{b(v, p)}{\| v \|_V} &\gtrsim \| p \|_b.  \label{eq: inf-sup B}
	\end{align}
\end{subequations}
Note the use of the norm $\| \cdot \|_b$ on $p$ in \eqref{eq: inf-sup B}. This distinguishes our analysis from the convention in which the norm $\| \cdot \|_Q$ is used in \eqref{eq: inf-sup B} instead, see e.g. \cite{boffi2013mixed,toselli2006domain}.

Finally, we introduce the parameter-dependent energy norm 
\begin{align} \label{eq: energy norm}
	\enorm{(v, q)}^2 
	:= \| v \|_V^2 + \| q \|_b^2 + t^2 | q |_c^2.
\end{align}

Note that $t$ in this norm reflects the additional stability obtained from the $C$-block. For $t = 0$, we obtain stability directly from the Brezzi conditions. However, for the range $t \in [0, 1]$, we require an additional inf-sup condition, as presented in the following theorem.

\begin{theorem}[Brezzi-Braess] \label{thm: Braess}
	Let the bilinear forms $a$ and $b$ satisfy the Brezzi conditions \eqref{eqs: Brezzi}. If, moreover,
	\begin{align}
		\forall u &\in V, & 
		\sup_{(v, q) \in V \times Q} \frac{a(u, v) + b(u, q)}{\enorm{(v, q)}} &\gtrsim \| u \|_V,  \label{eq: Braess inf-sup}
	\end{align}
	then problem \eqref{eq:ABBC} is stable in the energy norm \eqref{eq: energy norm}.
\end{theorem}
\begin{proof}
	See \cite[Lemma 3]{braess1996stability}.
\end{proof}

In the following, we refer to \eqref{eq: Braess inf-sup} as the \emph{Braess} condition. We will now demonstrate the use of this result for an exemplary problem concerning the modified Poisson equation, followed by the four-field formulation of Biot's equations \eqref{eq:biot}.

\section{A generalized Poisson (or simplified Helmholtz) equation in mixed form}\label{sec:helmholtz}

\subsection{Dirichlet boundary conditions}\label{sec:helmholtz_dirichlet}
Let us consider the following elliptic problem (here referred to as generalized Poisson equation, or also \emph{modified/simplified} Helmholtz equation because the squared wavenumber is taken with the opposite sign):
\begin{equation}\label{eq:helmholtz_dirichlet_strong}
\begin{aligned}
	-\nabla\cdot(K\nabla p)+\alpha p &= f \qquad 
	\text{in }\Omega, &
	p &= 0 \qquad
	\text{on }\partial \Omega,
\end{aligned}
\end{equation}
with coefficient matrix $K$, prescribed right-hand side $f$, and scalar parameter $0 \le \alpha \le 1$ (the squared wavenumber); 
and its  mixed formulation in operator form, given by: \\
Find $\bu \in V$ and $p \in Q$ such that
\begin{equation}\label{eq:Helmholtz}
\begin{bmatrix}
K^{-1}I & \nabla\\
\nabla\cdot & \alpha I\\
\end{bmatrix}\begin{bmatrix}
  \bu \\
  p
\end{bmatrix}=\begin{bmatrix}
0\\
f
\end{bmatrix}.
\end{equation}

Note that this problem has the structure of \eqref{eq:ABBC} (see similar mixed and mixed-hybrid formulations using Raviart-Thomas elements in, e.g., \cite{chaumon19,monk10}). We now define the appropriate function spaces and energy norm \eqref{eq: energy norm} using the properties of the operators $A$, $B$ and $C$. 

Starting with the Brezzi conditions \eqref{eqs: Brezzi}, we follow the theory presented in \cite{baerland2018observation} and consider the spaces
\begin{align}
	V &:= K^{-\frac12}\bL^2(\Omega) \cap \bH(\vdiv,\Omega), &
	Q_b &:= K^\frac12 H^1(\Omega) + L^2(\Omega).
\end{align}
These intersection and summation spaces are defined by the parameter-dependent norms
\begin{subequations}
\begin{align}
	\| \bu \|_V^2 &:= \| K^{-\frac12} \bu \|_{0,\Omega}^2 + \| \nabla \cdot \bu \|_{0,\Omega}^2, \label{eq:helmholtz-vnorm} \\
	\| p \|_b^2 &:= \inf_{r \in K^{\frac12}H^1(\Omega)} \biggl( \| p - r \|_{0,\Omega}^2 + \| K^\frac12 \nabla r \|_{0,\Omega}^2 \biggr).
\end{align}\end{subequations}
As shown in \cite{baerland2018observation}, both the inf-sup \eqref{eq: inf-sup B} and coercivity \eqref{eq: coercivity A} conditions hold in these norms. For more information on summation spaces, we refer the reader to \cite{bergh2012interpolation}. 

Letting $\alpha$ play the role of $t^2$ from Section~\ref{sec:abstract}, we have $C = I$ and thus
\begin{align*}
	\| p \|_c^2 &:= 
	c(p, p)
	= \| p \|_{0,\Omega}^2,
\end{align*}
and we remark that $\|p\|_c \not\lesssim \|p\|_b$.
However, we have $Q := L^2(\Omega) \cap Q_b = L^2(\Omega)$ and it remains to show that $Q$ is dense in $Q_b$. But this is immediate from the fact that $Q^{\perp_{Q_b}} = \{ 0\}$.

\begin{lemma} \label{lem: Helmholtz well-posed}
	Given the spaces $V$ and $Q$, their associated norms, and the bilinear forms in \eqref{eq:Helmholtz}, then the assumptions of Theorem~\ref{thm: Braess} hold. In turn, the problem is stable in the energy norm \eqref{eq: energy norm}.
\end{lemma}
\begin{proof}
	The validity of the Brezzi conditions \eqref{eqs: Brezzi} is shown in \cite{baerland2018observation}. Hence, it remains to show the Braess condition \eqref{eq: Braess inf-sup} and we proceed as follows. Given $\bu \in V$, let $\bv = \bu \in V$ and $q = \nabla \cdot \bu \in Q$. It follows that
	\begin{align*}
		a(\bu, \bv) + b(\bu, q)
		&= \| K^{-\frac12} \bu \|_{0,\Omega}^2 + \| \nabla \cdot \bu \|_{0,\Omega}^2
		= \| \bu \|_V^2.
	\end{align*}
	Moreover, we have
	\begin{align*}
		\enorm{(\bv, q)}^2
		:= \| \bv \|_V^2 + \| q \|_b^2 + \alpha \| q \|_c^2
		\lesssim \| \bu \|_V^2 + \| \nabla \cdot \bu \|_{0, \Omega}^2
		\lesssim \| \bu \|_V^2.
	\end{align*}
\end{proof}

\begin{remark}\label{rmrk:vv}
	For scalar $K$, it seems natural to consider the norms $\| \bv \|_V^2 := \| K^{-\frac12} \bv \|_{0, \Omega}^2 + \| \nabla \cdot K^{-\frac12} \bv \|_{0, \Omega}^2$ and $\| q \|_b := \| K^{\frac12} q \|_{0, \Omega}$ instead, similar to \cite{vassilevski2013block}. Following the proof of Lemma~\ref{lem: Helmholtz well-posed}, we would then choose $\bv = \bu$ and $q = - \nabla \cdot K^{-1} \bu$, such that 
	\begin{align}
		a(\bu, \bv) - b(\bu, q)
		&= \| K^{-\frac12} \bu \|_{0,\Omega}^2 + \| \nabla \cdot K^{-\frac12}\bu \|_{0,\Omega}^2
		= \| \bu \|_V^2.
	\end{align}
	However, for the second bound, we obtain
	\begin{align}\label{eq:aux00}
		\alpha \| q \|_c^2
		= \alpha \| \nabla \cdot K^{-1} \bu \|_0^2 
		\le \alpha K^{-1} \| \bu \|_V^2.
	\end{align}
	For the case $\alpha > 0$, the bound \eqref{eq:aux00} can not be improved with a constant independent of $K$. This explains the suboptimal performance of the preconditioner $\BB_{\text{VV}}$ for $K < \alpha$ observed in Table~\ref{tab:issues_vv} of Section \ref{sec:results}.
\end{remark}

\begin{remark}
	Strictly speaking, the bilinear form $c(\cdot, \cdot)$ is not continuous in $Q_b \times Q_b$ since the continuity constant would necessarily depend on $K$. For this reason, Theorem~\ref{thm: Braess} is more appropriate than the analysis presented in e.g. \cite{boffi2013mixed,toselli2006domain}, where continuity of $c$ is assumed.
\end{remark}


We remark that the energy norm is given by
\begin{align}
	\enorm{(\bu, p)}^2 &:= \| \bu \|_V^2 + \| p \|_b^2 + \alpha \| p \|_c^2 \nonumber\\
	&\eqsim
	\| \bu \|_V^2
	+
	\inf_{r \in K^{\frac12}H^1(\Omega)}
	\biggl( (1 + \alpha) \| p - r \|_{0,\Omega}^2 + 
	\alpha \| r \|_{0,\Omega}^2 + 
	\| K^\frac12 \nabla r \|_{0,\Omega}^2 \biggr).
\end{align}

Based on this norm, we use the theory from \cite{mardal-winther} to propose the following preconditioner for problem \eqref{eq:Helmholtz}:
\begin{equation}\label{eq:darcy_B}
\BB =
\begin{bmatrix}
  \left(K^{-1}I - \nabla\nabla\cdot\right)^{-1} & 0\\
  0 & \left((1 + \alpha) I\right)^{-1} + \left(\alpha I - K\Delta\right)^{-1}
\end{bmatrix}.
\end{equation}

\subsection{Neumann boundary conditions}\label{sec:helmholtz_neumann}

Consider the following generalized Poi\-sson problem with homogeneous Neumann boundary conditions for $0 \le \alpha \le 1$:
\begin{align} \label{eq: Helmholtz Neumann strong}
	\nabla \cdot (- \nabla p) + \alpha p &= f, &
	- \bnu \cdot \nabla p|_{\partial \Omega} &=0.
\end{align}
In the limit case of $\alpha = 0$, the solution $p$ is only defined for compatible $f$. We thus restrict this section to the case where $f$ has zero mean, i.e.
\[
	\bar{f} = \Pi_\mathbb{R} f = 0,
\]
with $\Pi_\mathbb{R}$ the projection onto constants. Applying this projection to the original equation, we immediately obtain that $\alpha \bar{p} = 0$. In the limit case, we have the freedom to choose $p$ with zero mean so this implies that $\bar{p} = 0$ for all $\alpha \ge 0$. This property is usually treated by searching the solution in the restricted function space $L^2(\Omega) / \mathbb{R}$. However, this can be cumbersome to discretize so we present an alternative approach, based on the observations from Section~\ref{sec:abstract}.

Let us consider the following, equivalent problem:\\ 
Find $\bu \in \bH_0(\vdiv,\Omega)$ and $p \in L^2(\Omega)$ such that
\begin{equation}\label{eq:Neumann mixed}
\begin{bmatrix}
I & \nabla\\
\nabla\cdot & \alpha I + (1 - \alpha) \Pi_\mathbb{R} \\
\end{bmatrix}\begin{bmatrix}
  \bu \\
  p
\end{bmatrix}=\begin{bmatrix}
0\\
f
\end{bmatrix}.
\end{equation}


\begin{lemma} \label{lem: Neumann well-posed}
	The solution $(\bu, p)$ to \eqref{eq:Neumann mixed} exists uniquely and satisfies
	\[
		\| \bu \|_{\vdiv,\Omega} + (1 + \sqrt{\alpha}) \| p \|_{0, \Omega} \lesssim \| f \|_{Q'}.
	\]
\end{lemma}
\begin{proof}
	We first consider existence. Letting $p$ be the solution to \eqref{eq: Helmholtz Neumann strong} and $\bu = - \nabla p$, it follows that $(\bu, p)$ solves \eqref{eq:Neumann mixed}. Uniqueness, on the other hand, follows by establishing the bound on the solution.

	We decompose the solution $p$ into its mean $\bar{p} \in \mathbb{R}$ and the deviation $\mathring{p} \in L^2(\Omega)/\mathbb{R}$. Let us consider these components separately. First, by applying $\Pi_{\mathbb{R}}$ to the second equation, we note that $\bar{p}$ is given by
	\begin{align*}
		\bar{p} = \bar{f} = 0.
	\end{align*}
	Secondly, $(\bu, \mathring{p}) \in \bH_0(\vdiv,\Omega) \times L^2(\Omega) / \mathbb{R}$ solves
	\begin{align} \label{eq: aux prob p0}
		\begin{bmatrix}
		I & \nabla\\
		\nabla\cdot & \alpha I \\
		\end{bmatrix}\begin{bmatrix}
		  \bu \\
		  \mathring{p}
		\end{bmatrix}=\begin{bmatrix}
		0\\
		\mathring{f}
		\end{bmatrix}.
	\end{align}

	This problem can be analyzed in the context of Theorem~\ref{thm: Braess}. We define the spaces $V := \bH_0(\vdiv,\Omega)$ and $Q_b = Q := L^2(\Omega) / \mathbb{R}$ and introduce the norms
	\begin{align}
		\| \bu \|_V &:= \| \bu \|_{\vdiv, \Omega}, &
		\| p \|_b &= | p |_c := \| p \|_{0, \Omega}.
	\end{align}
	
	The Brezzi conditions \eqref{eqs: Brezzi} are well-known to be satisfied for these spaces and norms. Moreover, the Braess condition \eqref{eq: Braess inf-sup} can be verified by assuming given $\bu$, setting $(\bv, \mathring{q}) = (\bu, \nabla \cdot \bu)$, and noting that
	\begin{subequations}
		\begin{align}
			a(\bu, \bv) + b(\bu, q) &= \| \bu \|_V^2, \\
			\enorm{ (\bv, q) }^2
			= \| \bu \|_V^2 + (1 + \alpha) \| \nabla \cdot \bu \|_{0, \Omega}^2
			&\lesssim \| \bu \|_V^2.
		\end{align}
	\end{subequations}
	Hence, the assumptions of Theorem~\ref{thm: Braess} are satisfied and we obtain the result by noting that $p = \mathring{p}$ and $f = \mathring{f}$.
\end{proof}

\begin{remark}
	Unlike \eqref{eq: Helmholtz Neumann strong}, problem \eqref{eq:Neumann mixed} does not require compatible data, and is uniquely solvable for all $f \in L^2(\Omega)$. The ``incompatibility'' is captured in the mean of $p$ since $\bar{p} = \bar{f}$. The stability result of Lemma~\ref{lem: Neumann well-posed} remains valid since we trivially have $\| \bar{p} \|_Q = \| \bar{f} \|_{Q'}$.
\end{remark}

As a direct consequence of Lemma~\ref{lem: Neumann well-posed}, we propose the following preconditioner for \eqref{eq:Neumann mixed}
\begin{align}\label{eq:Poisson_Neumann_precond}
	\BB = 
	\begin{bmatrix}
		(I - \nabla \nabla \cdot)^{-1} & 0 \\
		0 & ((1 + \alpha)I)^{-1}
	\end{bmatrix}.
\end{align}

Before returning to the problem of linear poroelasticity, we stress that a large class of problems can be put in the framework developed in this section. As an example, in Appendix~\ref{herrmann} we discuss the application into the discretization of Herrmann's formulation of linear elasticity \cite{herrmann1965elasticity}, where the additional unknown of solid pressure is added to avoid volumetric locking. 

\section{Back to the four-field formulation of Biot equations}\label{sec:biot}

At this point, we want to apply the same strategy as in Section \ref{sec:helmholtz}
to construct a preconditioner for the Biot system $\AA$ from \eqref{eq:biot} (endowed with the boundary and initial conditions \eqref{eqs: bdry and init conditions}). Let us consider the function spaces 
\begin{align*}
    V &:=  2\mu\bH^1_{\Gamma^{\bu}}(\Omega) \times [(\tau K)^{-\frac12} \bL^2(\Omega)\cap \bH_{\Gamma^{\bu}}(\vdiv,\Omega)], \\
    Q_b &:= \mu^{-1}L^2(\Omega) \times [(\tau K)^{\frac12} H^1(\Omega) + L^2(\Omega)], 
\end{align*}
and introduce the following (semi-)norms
\begin{subequations} \label{eq:norm-p1}
\begin{align}
	\| \vec \bu \|_V^2 &= 
	2 \mu \| \beps(\bu) \|_{0,\Omega}^2 
	+ (\tau K)^{-1} \| \bq_\tau \|_{0,\Omega}^2 
	+ \| \nabla \cdot \bq_\tau \|_{0,\Omega}^2, \\
	\| \vec p \|_b^2 &= 
	\mu^{-1} \| p_T \|_{0,\Omega}^2 + \| p \|_{L^2(\Omega) + (\tau K)^{\frac12} H^1(\Omega)}^2,\\
	| \vec p |_c^2 &=
	\frac{1}{\lambda} \| p_T + \alpha p \|_{0,\Omega}^2 
	+ c \| p \|_{0,\Omega}^2.
\end{align}
\end{subequations}

We define $Q$ as the subspace of $Q_b$ consisting of elements $\vec q$ with $| \vec q |_c < \infty$. Density of $Q$ in $Q_b$ follows once more from the fact that $Q^{\perp_{Q_b}} = \{ 0 \}$. 

The energy norm is given by \eqref{eq: energy norm}, and we repeat it here for convenience:
\begin{align}
	\enorm{(\vec \bv, \vec q)}^2 
	:= \| \vec \bv \|_V^2 + \| \vec q \|_b^2 + t^2 | \vec q |_c^2.
\end{align}
Since the scaling is not exactly given by a single parameter, we introduce $t$ as a scaling on the $C$-block and note that $t = 1$ corresponds to the original problem \eqref{eq:biot}. The limit case with $(\lambda, c) \to (\infty, 0)$ is then equivalent to setting $t = 0$. Since Theorem~\ref{thm: Braess} covers both cases, it forms the fundamental ingredient in our main result, presented in the following theorem.

\begin{theorem}\label{thm: biot well-posed}
	Problem \eqref{eq:biot} is well-posed and the solution $(\vec \bu, \vec p) \in V \times Q$ satisfies
	\begin{align}
		\| \vec \bu \|_V^2 + \| \vec p \|_b^2 + | \vec p |_c^2 \lesssim \| F \|_{V'} + \| G \|_{Q'}.
	\end{align}
\end{theorem}
\begin{proof}
	We show that the assumptions of Theorem~\ref{thm: Braess} are satisfied. Thus, let us consider the two Brezzi conditions \eqref{eqs: Brezzi} and the additional Braess condition \eqref{eq: Braess inf-sup}:
\begin{itemize}
	\item \emph{Coercivity of $A$ on $\Ker B$.}
	For $\vec \bv = (\bu, \bq_\tau)\in \Ker B$, we have $\nabla \cdot \bq_\tau = 0$. It then directly follows that
	\begin{align*}
		a(\vec \bv, \vec \bv) &=
		2 \mu \| \beps(\bu) \|_{0,\Omega}^2 
		+ (\tau K)^{-1} \| \bq_\tau \|_{0,\Omega}^2 
		= \| \vec \bv \|_V^2, & \forall \vec \bv \in \Ker B.
	\end{align*}
	\item \emph{Inf-sup of $B^{\tt{t}}$.}
	Let $\vec p = (p_T, p) \in Q_b$ be given. The usual inf-sup condition of Stokes problems, after a scaling by $\mu$, gives us that
	\begin{align*}
		\sup_{\bv} 
		\frac{		
		(\nabla \cdot \bv, p_T)_\Omega
		}{
		\| \mu^{\frac12} \bv \|_{1, \Omega}
		}
		&\gtrsim 
		\| \mu^{-\frac12} p_T \|_{0, \Omega}.
	\end{align*}
	Moreover, it was shown in \cite{baerland2018observation} that 
	\begin{align*}
		\sup_{\bq_\tau}
		\frac{(\nabla \cdot \bq_\tau, p)_\Omega}
		{\| \bq_\tau \|_{(\tau K)^{-\frac12}\bL^2(\Omega)\cap \bH(\vdiv,\Omega)} } 
		&\gtrsim 
		\| p \|_{L^2(\Omega) + (\tau K)^{\frac12} H^1(\Omega)}.
	\end{align*}
	Combining the above, we obtain
	\begin{align*}
		\sup_{\vec \bv}
		\frac{b(\vec \bv, \vec p)}
		{\| \vec \bv \|_V}
		&=
		\sup_{\vec \bv}
		\frac{-(\nabla \cdot \bv, p_T)_\Omega + (\nabla \cdot \bq_\tau, p)_\Omega}
		{\| \vec \bv \|_V}
		\gtrsim
		\| \vec p \|_b.
	\end{align*}
	\item \emph{Inf-sup of $A + B$.}
	Let $\vec \bu = (\bu, \bq_\tau) \in V$ be given. We then choose $\vec \bv = \vec \bu$ and $\vec q = (0, \nabla \cdot \bq_\tau)$ to derive
	\begin{align*}
	&	a(\vec \bu, \vec \bv) + b(\vec \bu, \vec q)
		= 
		2 \mu \| \beps(\bu) \|_{0,\Omega}^2 
		+ (\tau K)^{-1} \| \bq_\tau \|_{0,\Omega}^2 
		+ \| \nabla \cdot \bq_\tau \|_{0,\Omega}^2
		= \| \vec \bu \|_V^2, \\
		&\enorm{(\vec \bv, \vec q)}^2 
		=
		\| \vec \bu \|_V^2 
		+ \| \nabla \cdot \bq_\tau \|_{L^2(\Omega) + (\tau K)^{\frac12} H^1(\Omega)}^2
		+ \left(\frac{\alpha^2}{\lambda} + c \right) 
		\| \nabla \cdot \bq_\tau \|_{0, \Omega}^2
		\\
		&\qquad \qquad \lesssim
		\| \vec \bu \|_V^2 
		+ \| \nabla \cdot \bq_\tau \|_{0, \Omega}^2 \\
		&\qquad \qquad \lesssim \| \vec \bu \|_V^2.
	\end{align*}
\end{itemize}
	The proof is finalized by invoking Theorem~\ref{thm: Braess} and noting that $t = 1$ forms a special case.
\end{proof}

Given that Theorem~\ref{thm: biot well-posed} establishes a parameter-robust stability, we can straightforwardly use the general approach from \cite{mardal-winther} to construct the following preconditioner involving the specific norms \eqref{eq:norm-p1}
\begin{equation}\label{eq:biot_preconditioner}
\BB = \begin{bmatrix}
  \left(-\bnabla\cdot(2\mu\beps)\right)^{-1} & 0 & 0 & 0 \\
  0 & \left((\tau K)^{-1}I - \nabla\nabla\cdot\right)^{-1} & 0 & 0 \\
  0 & 0 & \multicolumn{2}{c}{\multirow{2}{*}{$\mathcal{P}$}} \\
  0 & 0 & 
  \end{bmatrix},
\end{equation}
where the fluid and total 
pressure preconditioner, $\mathcal{P}$, is a $2\times 2$ operator defined as
\[
\mathcal{P} = \left(\begin{bmatrix} \frac{1}{\mu} I & 0\\ 0 & I\end{bmatrix} + \mathcal{C}\right)^{-1}
  +
  \left(\begin{bmatrix} \frac{1}{\mu} I & 0\\ 0 & -\tau K \Delta\end{bmatrix} + \mathcal{C}\right)^{-1},
\mbox{ and }\quad   
  \mathcal{C} = \begin{bmatrix}
		\frac{1}{\lambda}  & \frac{\alpha}{\lambda} \\
		\frac{\alpha}{\lambda} & c + \frac{\alpha^2}{\lambda}
\end{bmatrix}.
  \]
  We expect such preconditioner to be robust in the sense that the condition number of the (left-)preconditioned matrix $\BB \AA$ is bounded uniformly in the parameters $\{\mu, K, \tau, \lambda, \alpha, c\}$. 

\section{Discrete stability}
\label{sec:discrete}

\subsection{Abstract setting}
In order to define a finite element method, let $\cT_{h}$ be a
 conforming simplicial partition of 
$\bar\Omega$, constituted by tetrahedra (or triangles 
in 2D) $K$ of diameter $h_K$, with mesh size
$h:=\max\{h_K:\; K\in\cT_{h}\}$. The mesh is considered shape-regular. 
Given an integer $s\ge 0$ and a generic element
$K \in \cT_h$, the symbol $\mathbb{P}_s(K)$ will denote the 
space of polynomial functions defined locally on the element
$K$ and being of degree no greater than $s$.

For generic and conforming finite-dimensional subspaces $V_h\subset V$, $Q_h\subset Q$, let us consider the Galerkin scheme arising from the discretization of \eqref{eq:ABBtC}
\begin{subequations}\label{eq:discrete-abstract}
\begin{align}
    a(u_h,v_h)  +b(v_h,p_h) & = F(v_h) \qquad \forall  v_h\in V_h,\\
    b(u_h,q_h) - t^2 c(p_h,q_h)& = -G(q_h) \qquad \forall q_h\in Q_h,
\end{align}\end{subequations}
for which the following direct consequence of Theorem~\ref{thm: Braess} holds.
\begin{lemma}\label{lem: discrete-stab} Assume that $V_h,Q_h$ associated with adequate norms $\|\cdot\|_{V_h}$, $\|\cdot\|_{Q_{b,h}}$, $\enorm{(\cdot,\cdot)}_{h}$, fulfill the following discrete counterparts of the Brezzi conditions 
\begin{subequations} 
	\begin{align}
		\forall v_h &\in \Ker B_h, & 
		a(v_h, v_h) &\gtrsim \| v_h \|_{V_h}^2,  \label{eq: brezzi-h-1}\\
		\forall p_h &\in Q_{b,h}, &
		\sup_{v_h \in V_h} \frac{b(v_h, p_h)}{\| v_h \|_{V_h}} &\gtrsim \| p_h \|_{Q_{b,h}},  \label{eq:  brezzi-h-2}
	\end{align}
\end{subequations}
together with the discrete analogue of the Braess condition \eqref{eq: Braess inf-sup}, 
	\begin{align}
		\forall u_h &\in V_h, & 
		\sup_{(v_h, q_h) \in V_h \times Q_h} \frac{a(u_h, v_h) + b(u_h, q_h)}{\enorm{(v_h, q_h)}_h} &\gtrsim \| u_h \|_{V_h},  \label{eq: braess-h}
	\end{align}
all \eqref{eq: brezzi-h-1}, \eqref{eq:  brezzi-h-2}, and \eqref{eq: braess-h} with constants independent of 
the mesh size $h$ and of the perturbation parameter $t$. Then there exists a unique solution  
$(u_h,p_h) \in V_h \times Q_h$
to \eqref{eq:discrete-abstract}, which is stable in the energy norm $\enorm{(\cdot,\cdot)}_{h}$.   
\end{lemma}

\subsection{Discrete generalized Poisson  problem}\label{sec:discrete_helmholtz}
According to Lemma~\ref{lem: discrete-stab}, the stability of discretizations to \eqref{eq:Helmholtz} holds as long as 
the discrete Brezzi conditions and the discrete Braess condition are met by the chosen finite-dimensional spaces $V_h,Q_h$.  

Following \cite{baerland2018observation}, 
let us first denote by $\widehat{V}_h$ the space of $\bH(\vdiv,\Omega)$-conforming vector functions approximated using Brezzi-Douglas-Marini elements of order $s+1$ \cite{bdm85}, or Raviart-Thomas elements of order $s$ \cite{rt77}, and let $\widehat{Q}_h$ be the  space of discontinuous piecewise polynomials of order $s$. Then we can use the discrete gradient operator $\nabla_h:\widehat{Q}_h \to \widehat{V}_h$ defined by 
\begin{equation}\label{eq:discrete_grad}
(\nabla_h q_h, \bv_h ) = - (q_h, \nabla\cdot \bv_h),
\end{equation}
to define a discrete $H^1$-norm as $\| \nabla_h q_h \|_{0,\Omega}$. This suffices to construct the approximation space for the potential as 
\begin{equation}\label{eq:Qh-Helmholtz}
 Q_{b,h} = L_h^2 + K^{\frac12} H_h^1,
 \end{equation}
where $L_h^2$ corresponds to $\widehat{Q}_h$ equipped with the usual $L^2$-norm and $H_h^1$ denotes the space conformed by the set $\widehat{Q}_h$ in combination with the discrete $H^1$-norm. Therefore the norm associated with 
\eqref{eq:Qh-Helmholtz} is 
\begin{equation*}
\| q_h \|^2_{Q_{b,h}} = \inf_{r_h\in \widehat{Q}_h} \biggl( \| q_h - r_h\|_{0,\Omega}^2 + \| K^{\frac12} \nabla_h r_h\|_{0,\Omega}^2\biggr).
\end{equation*}
On the other hand, for the flux we consider 
\begin{equation}\label{eq:Vh-Helmholtz}
V_h = K^{-\frac12} \bL_h^2 \cap \bH_h(\vdiv),
\end{equation}
where $\bL_h^2$ corresponds to $\widehat{V}_h$ equipped with the usual $\bL^2$-norm and $\bH_h(\vdiv)$ denotes the space conformed by the set $\widehat{V}_h$ in combination with the usual $\bH(\vdiv)$-norm. Since $V_h \subseteq V$, we endow $V_h$ with the $V$-norm \eqref{eq:helmholtz-vnorm}. 

It is possible to verify that these spaces satisfy the conditions of Lemma~\ref{lem: discrete-stab}. For instance, notice that if we adapt the proof of the continuous result in Lemma~\ref{lem: Helmholtz well-posed} to the discrete setting, we need that $\nabla\cdot \bu_h \in Q_h$, which holds for the chosen pairs of mixed finite element spaces.
 
\subsection{Discrete mixed Biot consolidation system}
For the case of Biot equations, considering again Lemma~\ref{lem: discrete-stab} but in the context of Theorem~\ref{thm: biot well-posed}, we can identify conditions for the discrete solvability and robust stability. It turns out that the approximation spaces for displacement and  total pressure need to be inf-sup stable in the sense of the Brezzi conditions, and also the pair of spaces for filtration velocity and fluid pressure need to satisfy the discrete Brezzi conditions plus the additional requirement that $(0,\nabla\cdot \bq_{\tau,h}) \in Q_{b,h}$ (however it is not required that the divergence of discrete displacements belongs to the space of discrete total pressures).  

As feasible choices for the approximation spaces for two-dimensional problems, we therefore take overall continuous and vector-valued, piecewise polynomials of degree $s+2$ to approximate displacements (denoted $\widetilde{V}_h$), and discontinuous and piecewise polynomials of degree $s$ for the total pressure (denoted $\widehat{Q}_h$, as before).
Alternatively, we may choose the Taylor-Hood pair for both two and three-dimensional problems.
Likewise, discrete inf-sup stability is required for the fluid flux-pressure pair, for which we consider  $\bH(\vdiv,\Omega)$-conforming discretizations of the percolation fluxes using Brezzi-Douglas-Marini elements of order $s+1$ or Raviart-Thomas elements of order $s$ (denoted $\widehat{V}_h$ as before), and piecewise polynomials (overall discontinuous) of degree $s$ for the fluid pressure. 

Similar to the modified Helmholtz system, in the spaces 
\begin{equation}\label{eq:Vh-Qh-Biot}
V_h = 2\mu \bH_{\Gamma^{\bu},h}^1\times \left[ (\tau K)^{-\frac12} \bL_h^2 \cap \bH_{\Gamma^{\bu},h}(\vdiv)\right], \quad 
Q_{b,h} = \mu^{-1} L^2_h \times \left[(\tau K)^{\frac12} H^1_h + L_h^2\right],
\end{equation}
the discrete space $\bH_{\Gamma^{\bu},h}^1$ corresponds to $\widetilde{V}_h$ (restricted to discrete functions vanishing on $\Gamma^{\bu}$) equipped with the usual $\bH^1$-norm and $\bH_{\Gamma^{\bu},h}(\vdiv)$ denotes the space conformed by the set $\widehat{V}_h$ (restricted to discrete functions with normal traces vanishing on $\Gamma^{\bu}$) in combination with the usual $\bH(\vdiv)$-seminorm. On $V_h$ we can use the $V$-norm defined in \eqref{eq:norm-p1}, while for $Q_h$ we employ the norm 
\begin{align*}
\| \vec{q}_h \|^2_{Q_h} := & \frac{1}{\mu} \|q_{T,h}\|_{0,\Omega}^2 + 
 \inf_{r_h\in \widehat{Q}_h} \biggl( \| q_h - r_h\|_{0,\Omega}^2 + \| K^{\frac12} \nabla_h r_h\|^2_{0,\Omega}\biggr) \\
 &\qquad + 
 \frac{1}{\lambda} \| q_{T,h} + \alpha q_h \|^2_{0,\Omega} + c\| q_h\|^2_{0,\Omega}.
\end{align*}

\section{Numerical results}\label{sec:results}
We demonstrate robustness of the proposed preconditioners by considering
spectra of the preconditioned systems. More precisely,
given problem operator $\AA$ and a preconditioner $\BB$,
we are interested in stability of the condition numbers $\lambda_{h, \max}/\lambda_{h, \min}$,
where $\lambda_{h, \max}$, $\lambda_{h, \min}$ are the largest and
smallest (in magnitude) eigenvalues of the generalized eigenvalue problem
$\AA_h x=\lambda_h \BB_h x$ with $\AA_h$, $\BB_h$
being the respective finite element approximations of the operators.

We remark that rather than the discrete $H^1$-norm defined in terms of the discrete 
gradient operator \eqref{eq:discrete_grad}, we use an equivalent (see \cite{Rusten1996InteriorPP}), 
more implementation-friendly, norm defined in terms of the bilinear form $\Delta_h:\widehat{Q}_h\times \widehat{Q}_h\rightarrow\mathbb{R}$ given by
\[
\Delta_h(p_h, q_h) = \sum_{K\in\mathcal{T}_h} \int_K \nabla p_h \nabla p_h\mathrm{d}x + 
\sum_{E\in\mathcal{E}_I} \int_{E} \frac{1}{\avg{h_E}}\jump{p_h}\jump{q_h} \mathrm{d}s  +
\sum_{E\in\mathcal{E}_D} \int_{E} \frac{1}{h_E} p_h q_h \mathrm{d}s.    
\]
Here, $\mathcal{E}_I$ are the interior facets of $\mathcal{T}_h$ while $\mathcal{E}_D$ 
are the external facets associated with pressure (Dirichlet) boundary conditions.  
The jump and average values of $p_h\in \widehat{Q}_h$ are defined as 
$\avg{p_h}=\tfrac{1}{2}(p_h|_{K^{+}}+p_h|_{K^{-}})$ and
$\jump{p_h}=p_h|_{K^{+}}-p_h|_{K^{-}}$, respectively, with $K^{\pm}$ the two elements that 
share the internal facet.

\subsection{Robust preconditioners for the generalized Poisson equation}
In the following we let $\Omega=(0, 1)^2$ and $\mathcal{T}_h$ is a uniform structured triangulation
of the domain. 
\subsubsection*{Dirichlet boundary conditions} Using the (stable) discretization
given by $\mathbb{R}\mathbb{T}_0$-$\mathbb{P}_0$ elements, the robustness of \eqref{eq:darcy_B}
for the Dirichlet problem \eqref{eq:helmholtz_dirichlet_strong} can be seen in Figure \ref{fig:helm_p_rt0}. To strengthen 
the numerical evidence, the experiments were carried out also with the 
lowest order Brezzi-Douglas-Marini and $\mathbb{R}\mathbb{T}_1$-$\mathbb{P}_1$
elements. The results are given in Figures \ref{fig:helm_p_bdm1},  \ref{fig:helm_p_rt1}
in Appendix \ref{sec:other_elements}.
\begin{figure}
  \begin{center}
    \includegraphics[width=\textwidth]{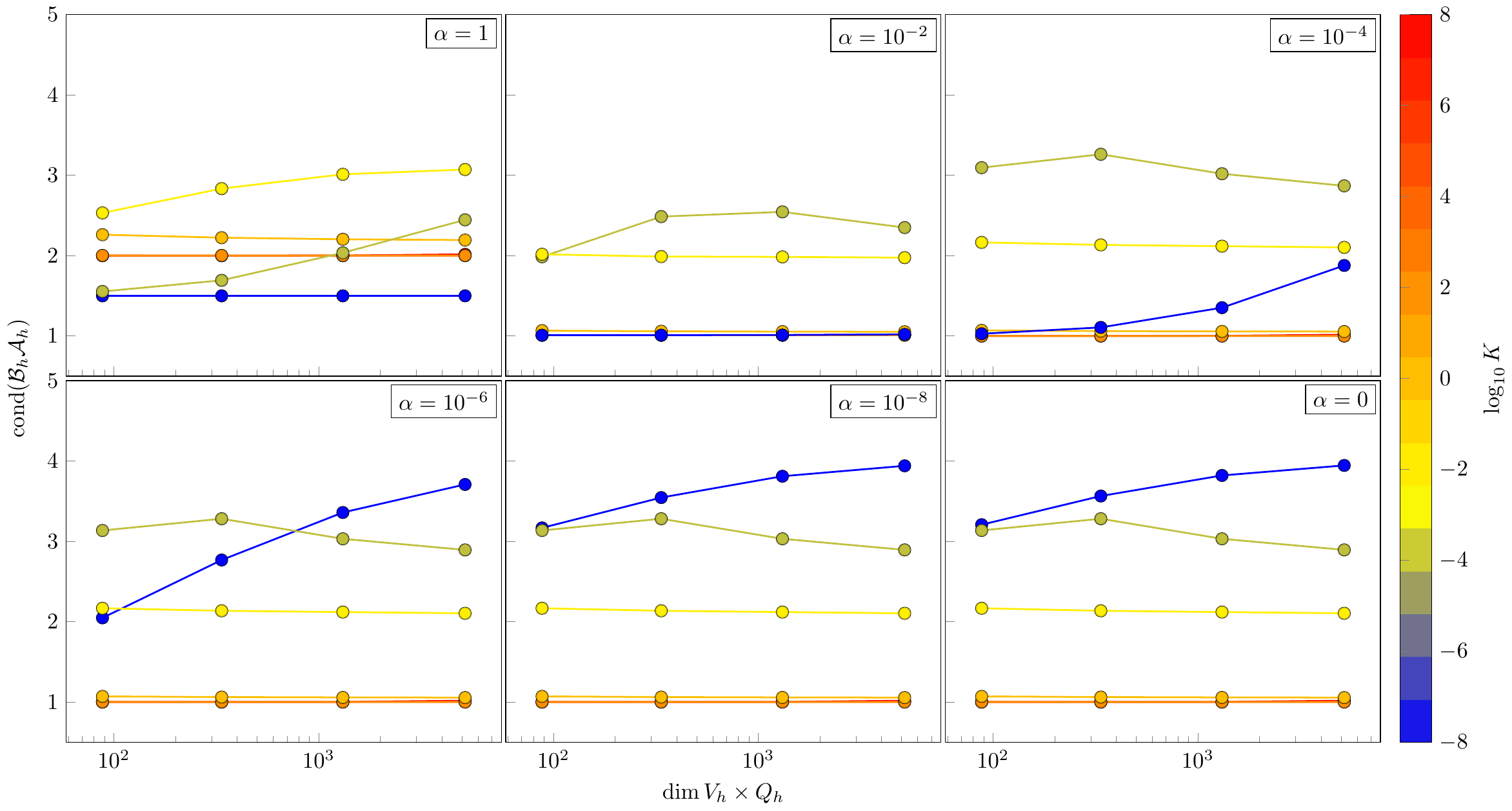}
    
    \vspace{-3mm}
    \caption{
      Performance of preconditioner \eqref{eq:darcy_B} for
      the generalized Poisson problem with pressure boundary conditions.
      Discretization by $\mathbb{R}\mathbb{T}_0$-$\mathbb{P}_0$ elements.
    }
    \label{fig:helm_p_rt0}
  \end{center}
\end{figure}

Before proceeding further, we  address two aspects of the analysis
in Section \ref{sec:helmholtz}. To compare \eqref{eq:darcy_B} with alternatives,
we recall a well-known $K$-robust preconditioner for the Darcy problem
(i.e. $\alpha=0$ in \eqref{eq:Helmholtz}) proposed in 
\cite{vassilevski2013block}. Extending it directly to the case of the generalized Poisson
problem, leads to 
\begin{equation}\label{eq:darcy_B_pannayot}
\BB_{\text{VV}}=
\begin{bmatrix}
  \left(K^{-1}I - K^{-1}\nabla\nabla\cdot\right)^{-1} & 0\\
  0 & \left(K I + \alpha I\right)^{-1}
\end{bmatrix}.
\end{equation}
However, as noted in Remark \ref{rmrk:vv}, Table \ref{tab:issues_vv}
shows that $K<\alpha$ leads to lack of robustness in $K$.
\begin{table}
  \begin{center}
      {
        \footnotesize{
\begin{tabular}{c|llll}
  \hline
  \multirow{2}{*}{$K$} & \multicolumn{4}{c}{$h$} \\
  \cline{2-5}
   & {$2^{-2}$ } & {$2^{-3}$} & {$2^{-4}$} & {$2^{-5}$}\\
\hline
$10^{-8}$ & 577 & 2306 & 9216 & --\\
$10^{-6}$ & 577 & 2300 & 9133 & --\\
$10^{-4}$ & 545 & 1874 & 4797 & 7867\\
$10^{-2}$ & 86 & 97 & 100 & 101\\
$10^{0}$ & 2.00 & 2.00 & 2.00 & 2.00\\
$10^{2}$ & 1.05 & 1.05 & 1.05 & 1.05\\
$10^{4}$ & 1.05 & 1.05 & 1.05 & 1.05\\
$10^{6}$ & 1.05 & 1.05 & 1.05 & 1.05\\
$10^{8}$ & 1.05 & 1.05 & 1.05 & 1.05\\
\hline
\end{tabular}
        }
        }
\end{center}    
\caption{Condition numbers of \eqref{eq:darcy_B_pannayot} for the generalized Poisson (or modified
  Helmholtz) problem and $\alpha=1$. Pressure boundary conditions
  are prescribed and $\mathbb{R}\mathbb{T}_0$-$\mathbb{P}_0$ is used
  for discretization. Condition numbers exceeding $10^{4}$ are indicated
  as --.
  In contrast, our proposed preconditioner \eqref{eq:darcy_B} does perform robustly for $K < \alpha$, as shown in Figure \ref{fig:helm_p_rt0}.
}
\label{tab:issues_vv}
\end{table}

Finally, Table \ref{tab:issues_alpha} illustrates the necessity of the assumption
of a small perturbation, i.e. $\alpha\leq 1$, for stability of the preconditioner
\eqref{eq:darcy_B}. Indeed, by setting $\alpha=10^{2}$, the sensitivity of
the condition numbers for $K<1$ becomes evident.
\begin{table}
  \begin{center}
      {
        \footnotesize{
\begin{tabular}{c|llll}
  \hline
  \multirow{2}{*}{$K$} & \multicolumn{4}{c}{$h$} \\
  \cline{2-5}
   & {$2^{-2}$ } & {$2^{-3}$} & {$2^{-4}$} & {$2^{-5}$}\\
\hline
$10^{-8}$  & 1.99 & 1.99 & 1.99 & 1.99\\       
$10^{-6}$  & 1.99 & 1.99 & 2.01 & 2.06\\       
$10^{-4}$  & 2.09 & 2.43 & 3.78 & 9.02\\       
$10^{-2}$  & 11 & 38 & 96 & 158\\  
$10^{0}$   & 166 & 190 & 197 & 198\\           
$10^{2}$   & 151 & 151 & 151 & 152\\           
$10^{4}$   & 151 & 152 & 152 & 152\\           
$10^{6}$   & 151 & 152 & 152 & 152\\           
$10^{8}$   & 151 & 152 & 152 & 152\\           
\hline
\end{tabular}
        }
        }
\end{center}    
  \caption{Condition numbers 
    of \eqref{eq:darcy_B} for the generalized Poisson (modified Helmholtz) problem \eqref{eq:Helmholtz}, setting 
  $\alpha=10^2$. Pressure boundary conditions
  are prescribed and $\mathbb{R}\mathbb{T}_0$-$\mathbb{P}_0$ is used
  for discretization.
  These results demonstrate the necessity of assuming small perturbations, i.e. $\alpha \le 1$.}
\label{tab:issues_alpha}
\end{table}

\subsubsection*{Neumann boundary conditions} Parameter stability of
preconditioner \eqref{eq:Poisson_Neumann_precond} for the Neumann problem 
\eqref{eq:Neumann mixed} is illustrated in Table \ref{tab:helm_neumann_K1}.
\begin{table}
  \begin{center}
  {
    \footnotesize{
\begin{tabular}{c|llll}
  \hline
  \multirow{2}{*}{$\alpha$} & \multicolumn{4}{c}{$h$} \\
  \cline{2-5}
   & {$2^{-2}$ } & {$2^{-3}$} & {$2^{-4}$} & {$2^{-5}$}\\
\hline
0        & 1.10 & 1.10 & 1.10 & 1.10\\
$10^{-8}$ & 1.10 & 1.10 & 1.10 & 1.10\\
$10^{-6}$ & 1.10 & 1.10 & 1.10 & 1.10\\
$10^{-4}$ & 1.10 & 1.10 & 1.10 & 1.10\\
$10^{-2}$ & 1.10 & 1.10 & 1.10 & 1.10\\
1        & 2.00 & 2.00 & 2.00 & 2.00\\
\hline
\end{tabular}
}}
\end{center}
  \caption{Condition numbers obtained for the Neumann problem \eqref{eq:Neumann mixed} using the  
    preconditioner \eqref{eq:Poisson_Neumann_precond}.}
    \label{tab:helm_neumann_K1}
\end{table}

Regarding the Neumann problem \eqref{eq:Neumann mixed} as a special case of
\begin{equation}\label{eq:Neumann mixed K}
\begin{bmatrix}
K^{-1} I & \nabla\\
\nabla\cdot & \alpha I + (1 - \alpha) \Pi_\mathbb{R} \\
\end{bmatrix}\begin{bmatrix}
  \bu \\
  p
\end{bmatrix}=\begin{bmatrix}
0\\
f
\end{bmatrix},
\end{equation}
let us finally address the preconditioning of problem \eqref{eq:Neumann mixed K}. 
Combining the analysis developed in Sections \ref{sec:helmholtz_dirichlet} and
\ref{sec:helmholtz_neumann}, we propose 
\begin{equation}\label{eq:darcy_B_neumann}
\BB=
\begin{bmatrix}
  \left(K^{-1}I - \nabla\nabla\cdot\right)^{-1} & 0\\
  0 & \left(\alpha I + (1-\alpha)\Pi_{\mathbb{R}}+ I\right)^{-1} + \left(\alpha I +(1-\alpha)\Pi_{\mathbb{R}} -K\Delta\right)^{-1}
\end{bmatrix},
\end{equation}
as a preconditioner for \eqref{eq:Neumann mixed K}. The robustness of such a 
preconditioner is demonstrated in Figure \ref{fig:helm_flux_rt0}, and
exemplified further in Figure \ref{fig:helm_flux_bdm1}.

\begin{figure}
  \begin{center}
    \includegraphics[width=\textwidth]{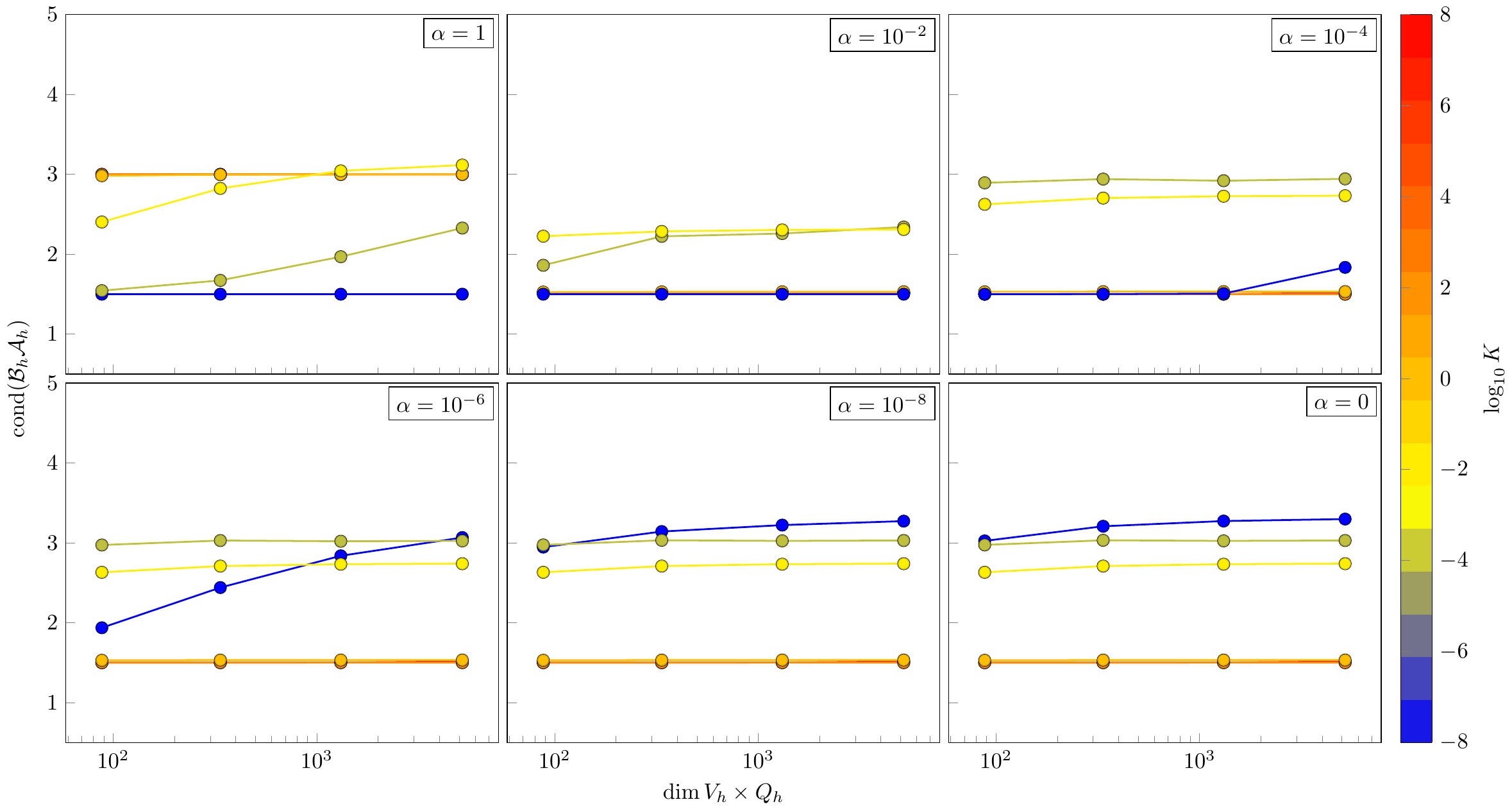}
    
    \vspace{-3mm}
    \caption{
      Performance of preconditioner \eqref{eq:darcy_B_neumann} for
      the generalized Poisson (simplified Helmholtz)  equation \eqref{eq:Neumann mixed K} with flux boundary conditions.
      Discretization by $\mathbb{R}\mathbb{T}_0$-$\mathbb{P}_0$ elements.
    }
    \label{fig:helm_flux_rt0}
  \end{center}
\end{figure}

\subsection{Robust preconditioners for the mixed Biot system}\label{sec:biot_experiments}
Due to its larger parameter space we restrict numerical
experiments for the Biot system to a single type of finite element discretization,
namely, we shall use continuous piecewise quadratic vector valued ($\mathbb{P}_2$)
functions for the displacement and lowest-order Raviart-Thomas
elements for the for the percolation velocity. The total 
pressure will be discretized using continuous piecewise linear 
Lagrange elements ($\mathbb{P}_1$) if $\Omega\subset\mathbb{R}^3$ while $\mathbb{P}_0$ 
is used in the two-dimensional case. Finally, the fluid pressure shall be
approximated by piecewise constants.

\subsubsection*{Parameter robustness} As in the case of the generalized Poisson
problem we consider a two-dimensional problem \eqref{eq:biot} with domain $\Omega=(0, 1)^2$, and
$\Gamma^{\bu}=\left\{(x, y)\in\partial\Omega: x=0\text{ or } x=1\right\}$,
$\Gamma^{\bsigma}=\partial\Omega\setminus\Gamma^{\bu}$. The domain is discretized
by a uniform mesh.

Figure \ref{fig:biot} shows variations of the condition numbers for the \eqref{eq:biot_preconditioner}-preconditioned Biot problem across the parameter ranges
$10^{-12}\leq K \leq 1$, $1\leq\lambda\leq 10^{16}$, $0\leq \alpha \leq 1$, $c < 0 \leq 1$.
It can be seen that the preconditioner yields bounded condition numbers (not exceeding
8 in the experiments). 
We remark that the time step $\tau$ was kept fixed at $\tau=1$ as its variations effectively 
translate to a modified hydraulic conductivity $K\tau$. 

\begin{figure}
  \begin{center}
    \includegraphics[width=\textwidth]{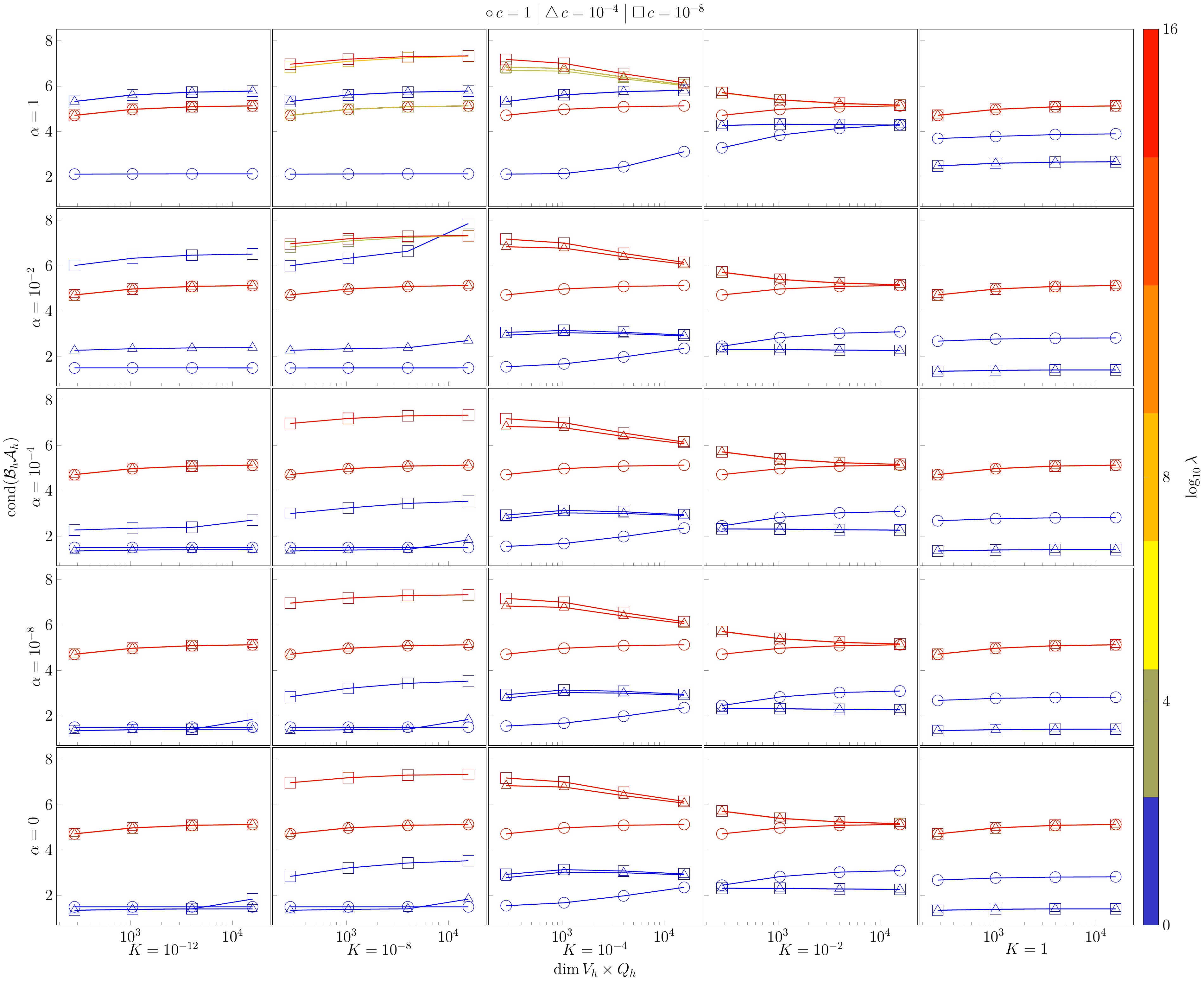}
    
    \vspace{-3mm}
    \caption{
      Performance of the preconditioner \eqref{eq:biot_preconditioner} for
      the Biot system \eqref{eq:biot}. Discretization
      by $\mathbb{P}_2$-$\mathbb{R}\mathbb{T}_0$-$\mathbb{P}_0$-$\mathbb{P}_0$
      elements. Values of $\alpha$ and $K$
      vary along the vertical, respectively horizontal axis, while color
      encodes Lam{\' e} constant $\lambda$, and storage capacity is depicted
      with markers. In most cases, the condition numbers for $\lambda > 1$ are
      very similar, leading to an overlap of the curves.
    }
    \label{fig:biot}
  \end{center}
\end{figure}

\subsubsection*{Scalable realization of the preconditioner}
Numerical experiments presen\-ted thus far have utilized the exact Biot preconditioner,
that is, each of the blocks was computed
by LU factorization. As such a construction is of limited interest in practical/large scale applications
we next briefly discuss the realization of \eqref{eq:biot_preconditioner} in terms
of off-the-shelf scalable components.

Indeed, the displacement block of the preconditioner is a standard operator
which can be efficiently realized by, e.g., multigrid \cite{brandt1986algebraic, vanvek1996algebraic}.
Similarly, geometric (see \cite{arnold2000multigrid}) and algebraic (see \cite{hdiv_amg}) multigrid methods
have been proposed for the Riesz map with respect to the weighted $\bH(\mbox{div})$-inner product, that in our context corresponds to  the flux preconditioner in \eqref{eq:biot_preconditioner}. These
methods have been shown to be robust in the respective  parameters (cf. experiments in \cite{hdiv_amg} for algebraic and
\cite[Section 4.1]{farrell2019pcpatch} for the geometric multigrid case).
Finally, to the best of the authors' knowledge, the approximation of the pressure block $\mathcal{P}$
of the Biot preconditioner has not been studied in literature. In this case, as the operator consists
of two inverses of symmetric elliptic operators, we expect multigrid methods to perform well. We remark, however, that the approximation might not be robust with respect to the model parameters.  

In order to illustrate the performance of the multigrid realization of the Biot preconditioner 
we consider the 3$d$ footing problem (see, e.g., \cite[Section 5.2.2]{gaspar2008distributive}) 
and set $\alpha=c=1/2$. Here the displacement preconditioner and
the components of $\mathcal{P}$ employ a single V-cycle of algebraic multigrid
(implemented in Hypre \cite{falgout2002hypre}) while the flux preconditioner is realized
with geometric multigrid using a hierarchy of three meshes in combination with the star  
smoother \cite{arnold2000multigrid}. The implementation has been carried out using the 
PCPATCH framework \cite{farrell2019pcpatch}.

Fixing the time step to $\tau=0.1$, Table \ref{tab:footing} displays the number of MinRes iterations
needed to reduce the preconditioned residual norm by a  factor of $10^6$ at each step
of the simulation. Taking the coarsest mesh for comparison, the use of multigrid 
approximately doubles the number of solver iterations as  compared to the exact
preconditioner. However, the iterations appear to be bounded in the mesh size. 
Samples of the approximate solution at the final time can be seen in Figure \ref{fig:footing}.

\begin{table}
  \begin{center}
  \footnotesize{
\begin{tabular}{cc|ccccc}
  \hline
  \multirow{2}{*}{$h^{-1}$} & \multirow{2}{*}{$\dim V_h\times Q_h$} & \multicolumn{5}{c}{$t$}\\
  \cline{3-7}
  & & 0.1 & 0.2 & 0.3 & 0.4 & 0.5\\
  \hline
16 & $1.88\times 10^{5}$ & 85(27) & 92(31) & 36(16) & 36(16) & 36(16)\\
32 & $1.46\times 10^{6}$ & 92 & 99 & 38 & 39 & 39\\
48 & $4.86\times 10^{6}$ & 101 & 110 & 41 & 41 & 41\\
\hline
\end{tabular}
  }
  \end{center}
\caption{
  Number of MinRes iterations obtained by using a  multigrid realization of the Biot preconditioner
  \eqref{eq:biot_preconditioner} for the 3$d$ footing problem. For the coarsest
  mesh the number in brackets indicates the iteration count using an exact preconditioner
  with blocks computed by a direct solver. Discretization by
  $\mathbb{P}_2$-$\mathbb{R}\mathbb{T}_0$-$\mathbb{P}_1$-$\mathbb{P}_0$ elements.
}
\label{tab:footing}
\end{table}

\begin{figure}
  \begin{center}
    \includegraphics[width=0.32\textwidth]{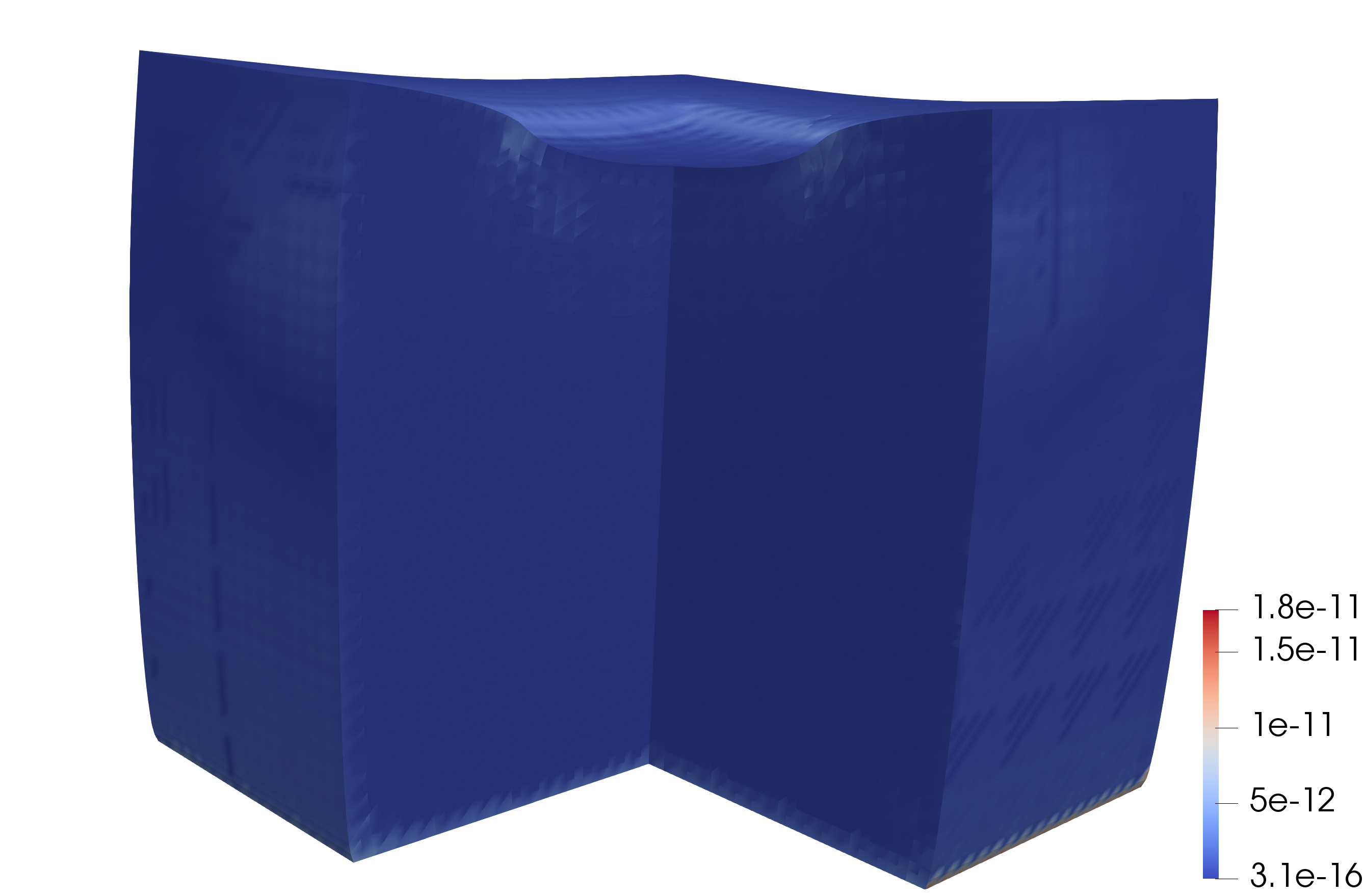}
    \includegraphics[width=0.32\textwidth]{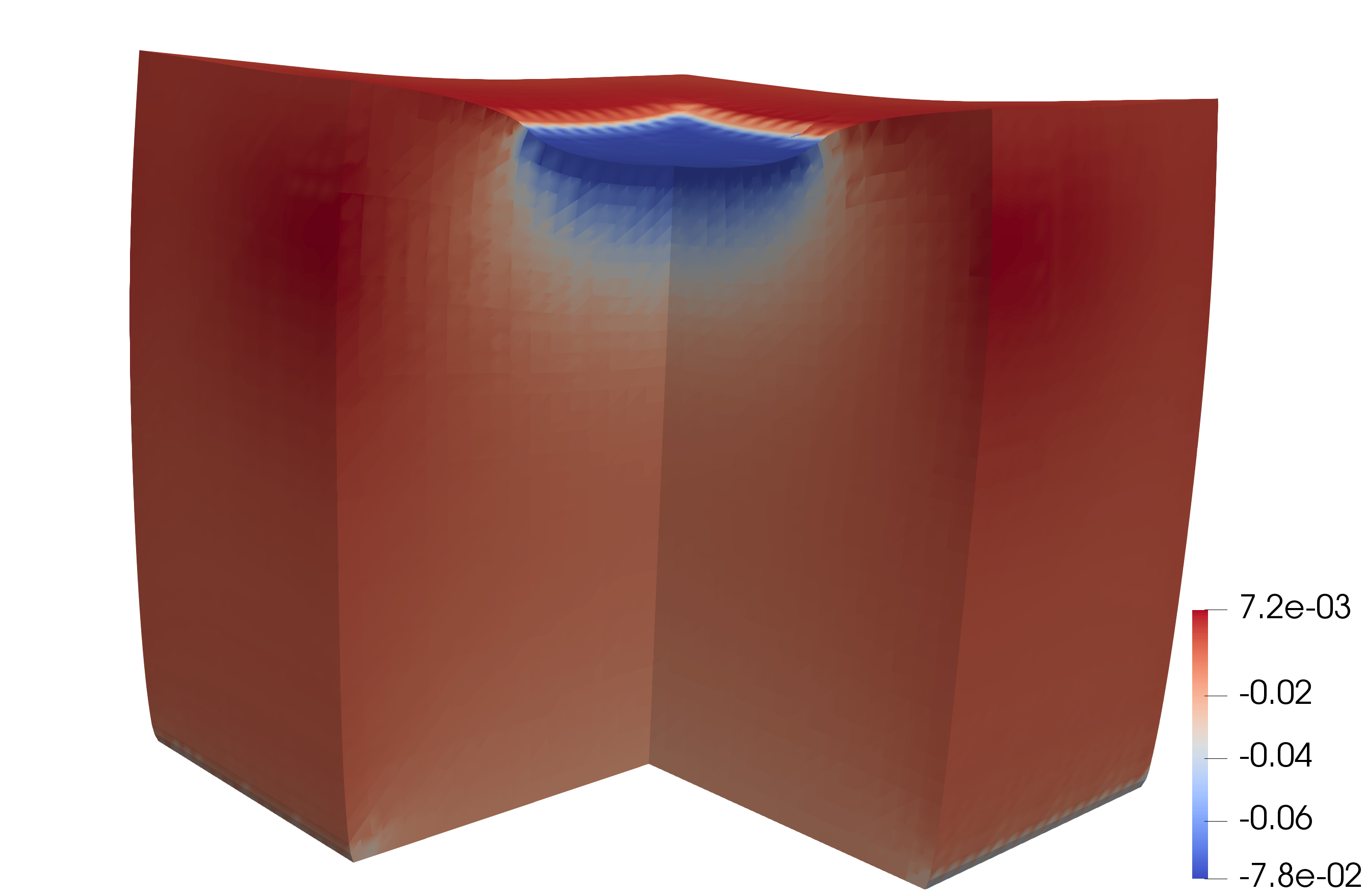}
    \includegraphics[width=0.32\textwidth]{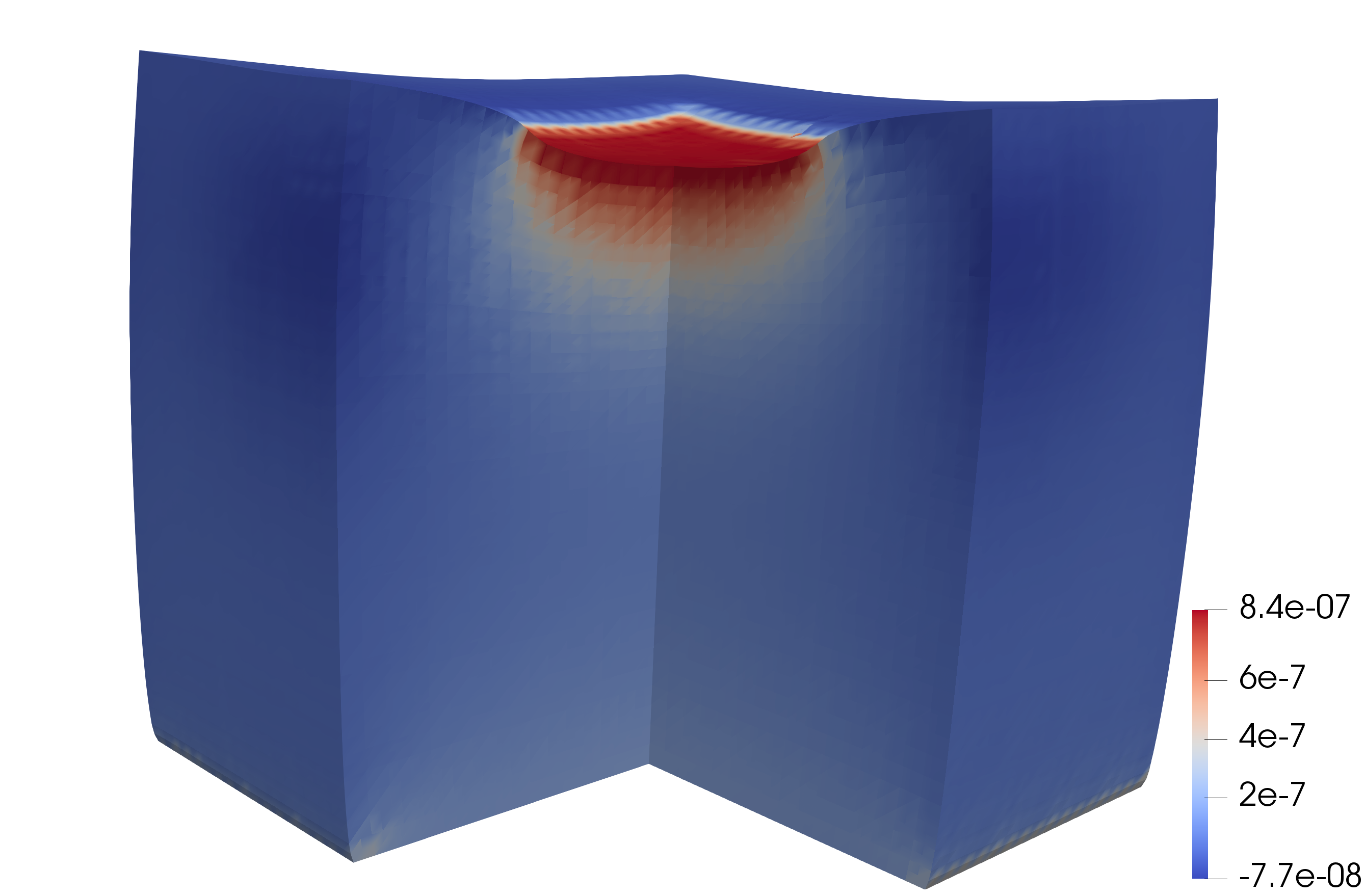}    
  \end{center}
  \caption{Approximate solutions for the footing problem at $t=0.5\,\text{s}$. 
    The percolation velocity magnitude (left), total pressure (center) and
    fluid pressure (right) are shown on a domain deformed by the computed 
    displacement (scaled by a factor of $10^6$).
  }
  \label{fig:footing}
\end{figure}
\bibliographystyle{siamplain}
\bibliography{references}

\bigskip
\appendix
\section{Stability of the  preconditioner for the generalized Poisson problem}\label{sec:other_elements}
This section presents results of numerical experiments showing robustness
of the proposed preconditioners when different (than in the main article
text) finite elements are used for the discretization.

\begin{figure}[h]
  \begin{center}
    \includegraphics[width=\textwidth]{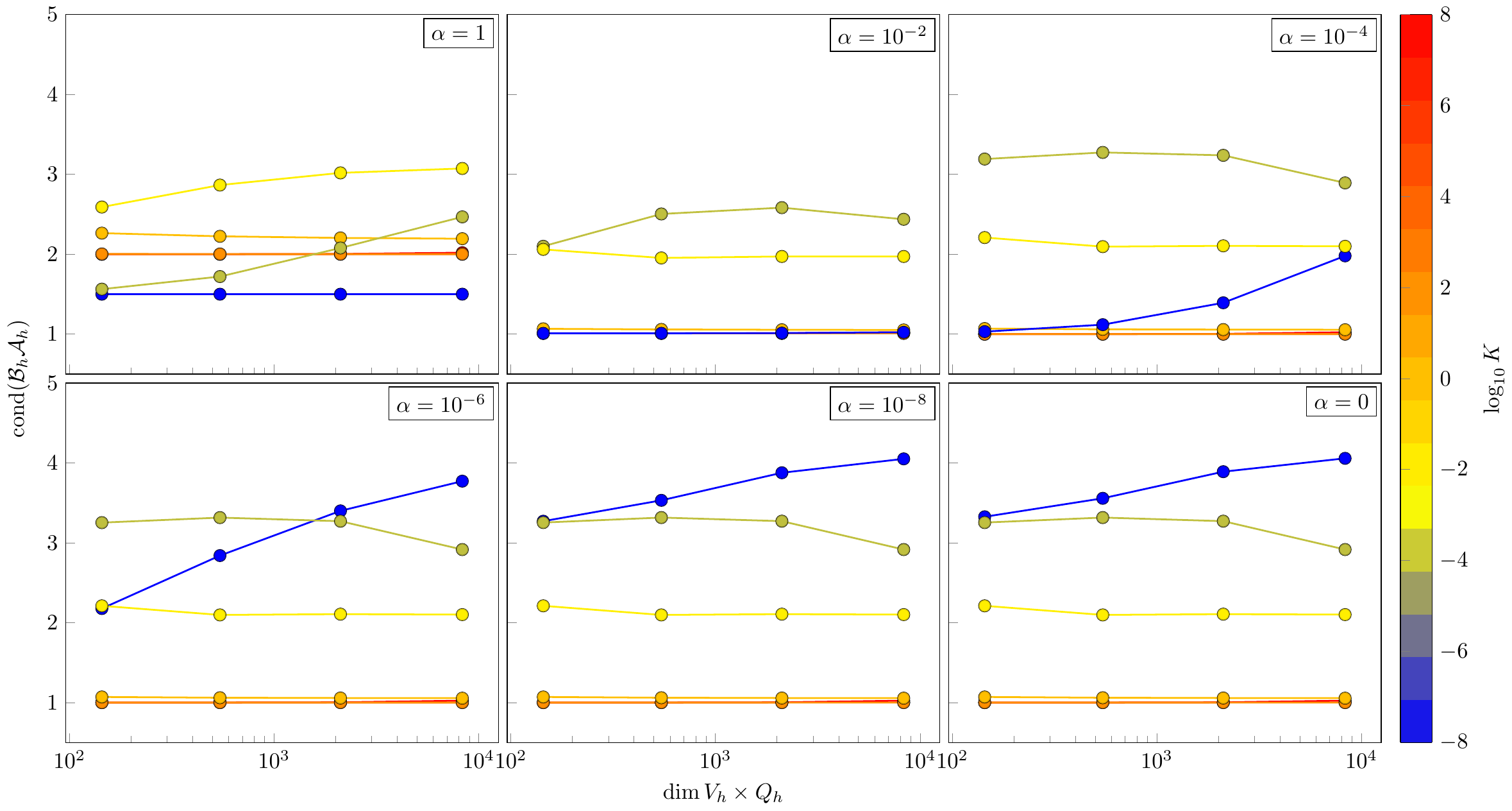}
    
    \vspace{-3mm}
    \caption{
      Performance of preconditioner \eqref{eq:darcy_B} for
      the simplified Helmholtz equation with potential boundary conditions.
      Discretization by $\mathbb{B}\mathbb{D}\mathbb{M}_1$-$\mathbb{P}_0$ elements.
    }
    \label{fig:helm_p_bdm1}
  \end{center}
\end{figure}

\begin{figure}
  \begin{center}
    \includegraphics[width=\textwidth]{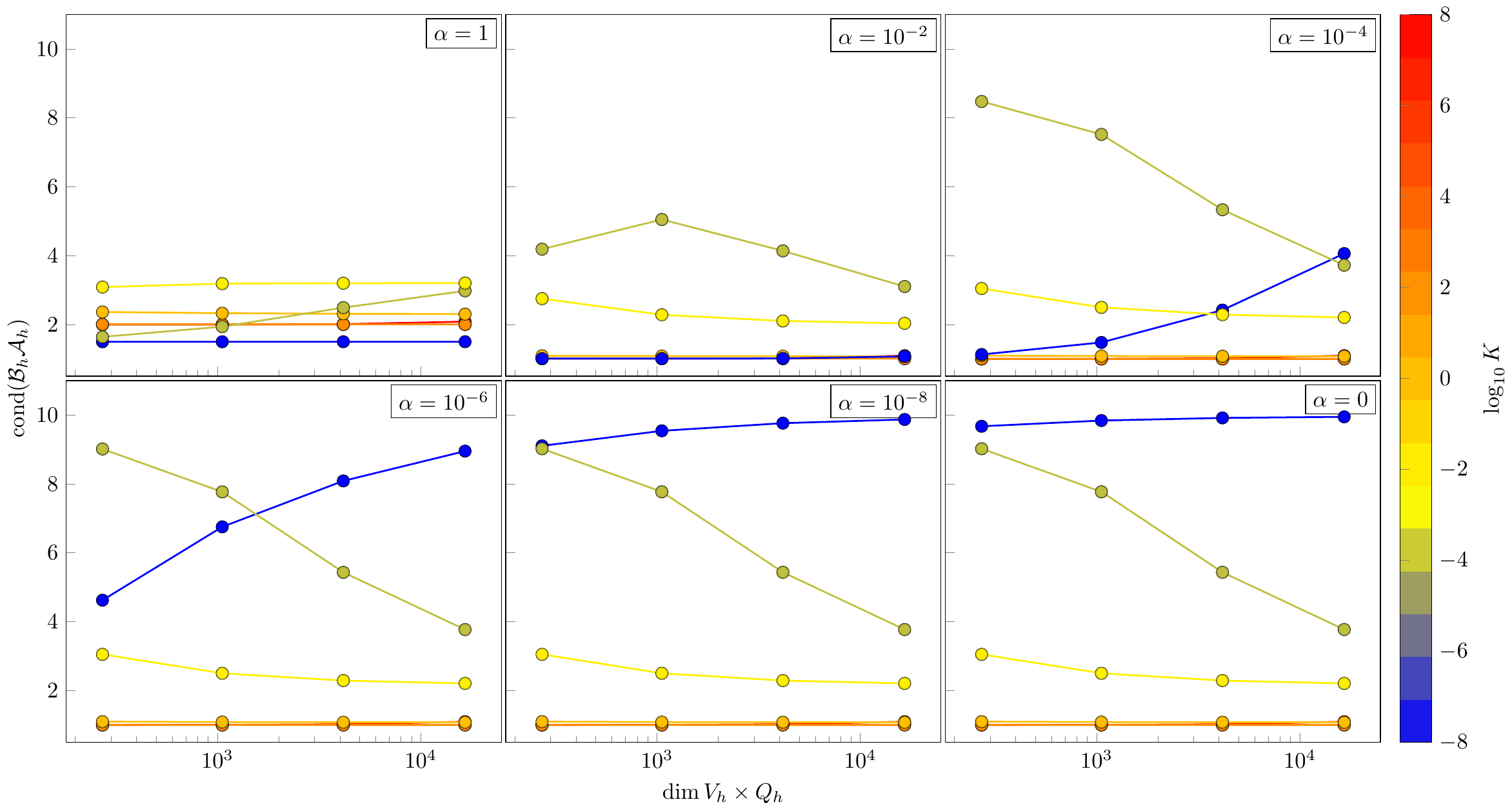}
    
    \vspace{-3mm}
    \caption{
      Performance of preconditioner \eqref{eq:darcy_B} for
      the generalized Poisson equation with potential boundary conditions.
      Discretization by $\mathbb{R}\mathbb{T}_1$-$\mathbb{P}_1$ elements.
    }
    \label{fig:helm_p_rt1}
  \end{center}
\end{figure}

\begin{figure}
  \begin{center}
    \includegraphics[width=\textwidth]{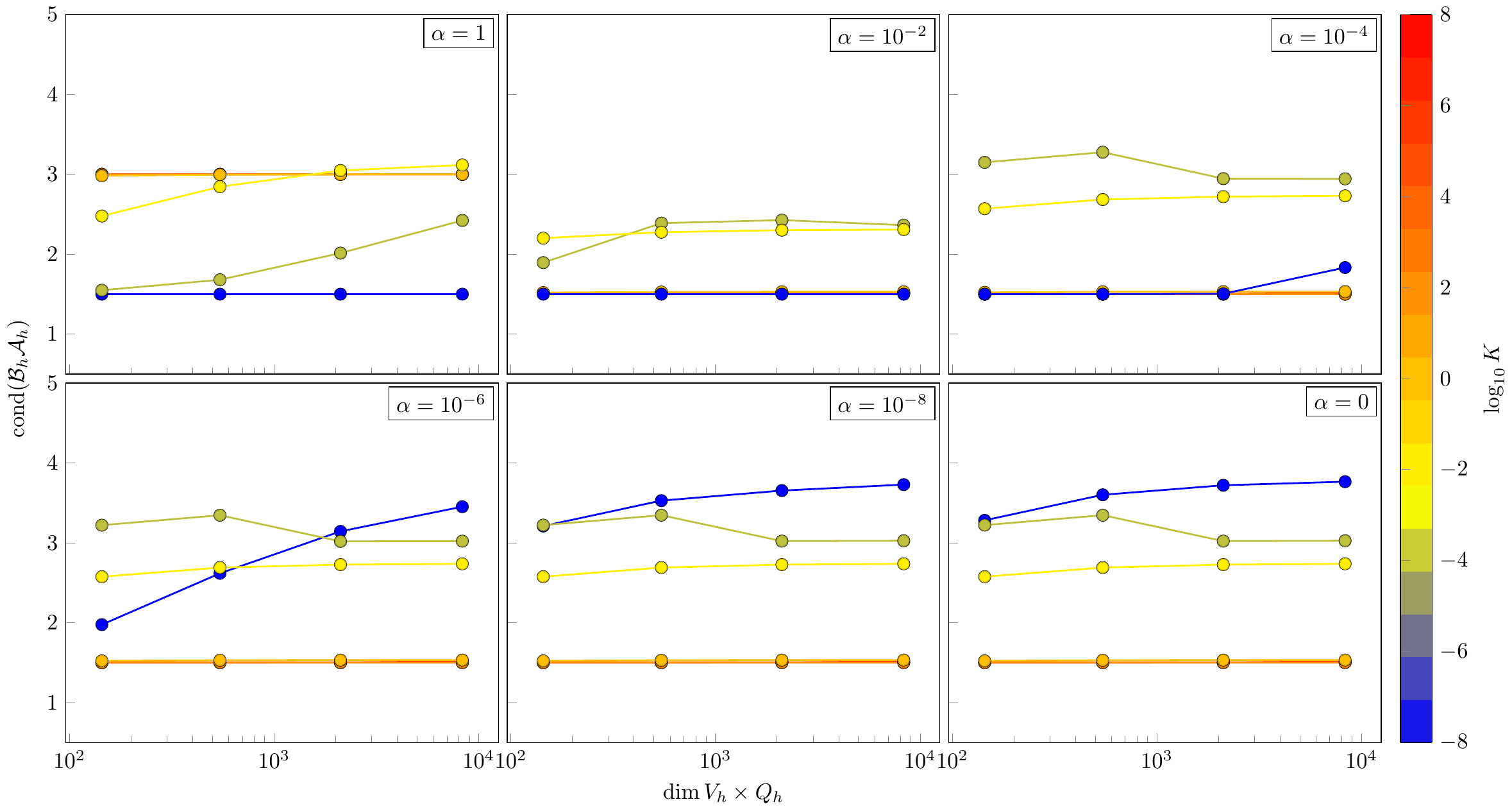}
    
    \vspace{-3mm}
    \caption{
      Performance of preconditioner \eqref{eq:darcy_B_neumann} for
      the generalized Poisson equation \eqref{eq:Neumann mixed K} with flux boundary conditions.
      Discretization by $\mathbb{B}\mathbb{D}\mathbb{M}_1$-$\mathbb{P}_0$ elements.
    }
    \label{fig:helm_flux_bdm1}
  \end{center}
\end{figure}

\section{Suboptimal preconditioners for Biot equations}
Here we present a few preconditioners that might be the natural and intuitive choices
if one starts from the original formulation  \eqref{eq:biot}, but that do not retain
robustness with respect to model parameters. For simplicity, we let
$\tau=1$ and $\mu=1$ and only focus on stability with respect
to the remaining model parameters. Numerical experiments then utilize
the same (two-dimensional) setup as the robustness study of Section \ref{sec:biot_experiments}.

First, one can suggest a preconditioner being the inverse of: 
\[
\BB_1=
 \begin{bmatrix}
-\bnabla\cdot(2\mu\beps)  & 0 & 0& 0\\
0 & K^{-1}I - \nabla\nabla\cdot & 0 & 0 \\
0 & 0 & \mu^{-1}I &  0 \\
0 & 0 & 0 & I 
\end{bmatrix}
 \]
 And one expects that the performance might be compromised for some
 combinations of the model parameters $\alpha$, $c$, $K$, $\lambda$, cf.~Table \ref{tab:biot_subopt_B1}. Therefore we consider other options. For
 instance, exploring preconditioners based on the inverse of

\begin{table}[ht]
\setlength{\tabcolsep}{3pt}
\begin{center}
{\footnotesize{
\begin{tabular}{c|c|c|lll}
  \hline
  \multirow{2}{*}{$c$} & \multirow{2}{*}{$K$} & \multirow{2}{*}{$\lambda$} & \multicolumn{3}{c}{$h$}\\
  \cline{4-6}
  & & & {$2^{-2}$ } & {$2^{-3}$} & {$2^{-4}$}\\
\hline  
  \multirow{9}{*}{1} & \multirow{3}{*}{$10^{-8}$} & 1 & 5.25 & 5.28 & 5.30\\
&  & $10^{3}$ & 6.75 & 7.08 & 7.22\\
  &  & $10^{9}$ & 6.77 & 7.11 & 7.25\\
  \cline{2-6}
&\multirow{3}{*}{$10^{-4}$} & 1 & 5.25 & 5.30 & 5.42\\
&  & $10^{3}$ & 6.75 & 7.08 & 7.22\\
  &  & $10^{9}$ & 6.77 & 7.11 & 7.25\\
  \cline{2-6}  
&\multirow{3}{*}{1} & 1 & 6.19 & 6.27 & 6.31\\
&  & $10^{3}$ & 7.13 & 7.23 & 7.28\\
&  & $10^{9}$ & 7.15 & 7.26 & 7.30\\
\hline
\end{tabular}}}
\hspace{10pt}
{\footnotesize{
\begin{tabular}{c|c|c|lll}
  \hline
  \multirow{2}{*}{$c$} & \multirow{2}{*}{$K$} & \multirow{2}{*}{$\lambda$} & \multicolumn{3}{c}{$h$}\\
  \cline{4-6}
  & & & {$2^{-2}$ } & {$2^{-3}$} & {$2^{-4}$}\\
  \hline
  \multirow{9}{*}{$10^{-2}$} & \multirow{3}{*}{$10^{-8}$} & 1 & 17.62 & 18.02 & 18.20\\
  &   & $10^{3}$ & 139 & 144 & 146\\
  &   & $10^{9}$ & 153 & 158 & 160\\
  \cline{2-6}
&\multirow{3}{*}{$10^{-4}$} & 1 & 17.53 & 17.91 & 18.17\\
  &   & $10^{3}$ & 128 & 132 & 134\\
  &   & $10^{9}$ & 139 & 144 & 146\\
  \cline{2-6}
&\multirow{3}{*}{1} & 1 & 3.91 & 3.94 & 3.95\\
  &   & $10^{3}$ & 6.75 & 7.08 & 7.22\\
  &   & $10^{9}$ & 6.77 & 7.11 & 7.25\\
\hline
\end{tabular}}}
\end{center}
\caption{Condition numbers of $\BB_1$-preconditioned Biot problem \eqref{eq:biot} with
  $\alpha=1$, $\mu=1$. The preconditioner seems robust in $K$ and $\lambda$ for $c=1$ but
  not for $c<1$.
}
\label{tab:biot_subopt_B1}
\end{table}
 \[
 \BB_2=
 \begin{bmatrix}
 -\bnabla\cdot(2\mu\beps) & 0 & 0 & 0\\
 0 & K^{-1}(I - \nabla\nabla\cdot) & 0 & 0 \\
 0 & 0 & \mu^{-1}I &  0\\
 0 & 0 & 0 & KI
 \end{bmatrix},
 \]
 or, alternatively, using 
 \[
 \BB_3=
 \begin{bmatrix}
-\bnabla\cdot(2\mu\beps)  & 0 &  0 & 0 \\
 0 & K^{-1}I - \nabla\nabla\cdot & 0 & 0  \\
 0 & 0 & \mu^{-1}I &  0\\
 0 & 0 & 0 & (I +  (- \nabla\cdot K \nabla)^{-1})^{-1} 
 \end{bmatrix}.
 \]
 Finally, we can consider the following modification of $\BB_2$:
 \[
 \BB_4=
 \begin{bmatrix}
\bnabla\cdot(2\mu\beps)  & 0 & 0 & 0\\
 0 &  K^{-1}(I - \nabla\nabla\cdot) & 0 & 0  \\
 0 & 0 & (\mu^{-1} + \frac{1}{\lambda})I & \frac{\alpha}{\lambda}I \\
 0 & 0 & \frac{\alpha}{\lambda}I & (K + c + \frac{\alpha^2}{\lambda})I
 \end{bmatrix}.       
 \]
 However these preconditioners yield sub-optimal performance,
 as evidenced in Table~\ref{tab:biot_subopt_Bother}.

 \begin{table}[h]
\setlength{\tabcolsep}{3pt}
\begin{center}
{\footnotesize{
\begin{tabular}{c|c|lll||lll||lll}
\hline
\multirow{2}{*}{$K$} & \multirow{2}{*}{$\lambda$} & \multicolumn{3}{c||}{$h$} & \multicolumn{3}{c||}{$h$} & \multicolumn{3}{c}{$h$}\\
  \cline{3-11}
 & & {$2^{-2}$ } & {$2^{-3}$} & {$2^{-4}$} & {$2^{-2}$ } & {$2^{-3}$} & {$2^{-4}$} & {$2^{-2}$ } & {$2^{-3}$} & {$2^{-4}$}\\
  \hline
\multirow{3}{*}{$10^{-8}$} & 1 & $1 \times 10^{11}$ & $5 \times 10^{11}$ & $2 \times 10^{12}$  &  $9 \times 10^{7}$ & $8 \times 10^{7}$ & $7 \times 10^{7}$ & 721 & $3 \times 10^{3}$ & $1 \times 10^{4}$\\
                    & $10^{3}$ & $5 \times 10^{10}$ & $2 \times 10^{11}$ & $9 \times 10^{11}$  &  $1 \times 10^{8}$ & $1 \times 10^{8}$ & $1 \times 10^{8}$ & 833 & $4 \times 10^{3}$ & $1 \times 10^{4}$\\
                    & $10^{9}$ & $5 \times 10^{10}$ & $2 \times 10^{11}$ & $9 \times 10^{11}$  &  $1 \times 10^{8}$ & $1 \times 10^{8}$ & $1 \times 10^{8}$ & 833 & $4 \times 10^{3}$ & $1 \times 10^{4}$\\
\hline
\multirow{3}{*}{$10^{-4}$} & 1 & $1 \times 10^{7}$ & $4 \times 10^{7}$ & $1 \times 10^{8}$ &  $9 \times 10^{3}$ & $8 \times 10^{3}$ & $7 \times 10^{3}$ & 693 & $3 \times 10^{3}$ & $7 \times 10^{3}$\\
                    & $10^{3}$ & $5 \times 10^{6}$ & $2 \times 10^{7}$ & $5 \times 10^{7}$ &  $1 \times 10^{4}$ & $1 \times 10^{4}$ & $1 \times 10^{4}$ & 790 & $3 \times 10^{3}$ & $8 \times 10^{3}$\\
                    & $10^{9}$ & $5 \times 10^{6}$ & $2 \times 10^{7}$ & $5 \times 10^{7}$ &  $1 \times 10^{4}$ & $1 \times 10^{4}$ & $1 \times 10^{4}$ & 790 & $3 \times 10^{3}$ & $8 \times 10^{3}$\\
\hline
\multirow{3}{*}{1} & 1 & 6.19 & 6.27 & 6.31 &   7.35 & 7.23 & 7.18 & 3.33 & 3.46 & 3.52\\
             & $10^{3}$ & 7.13 & 7.23 & 7.28 &  8.37 & 8.25 & 8.19 & 6.75 & 7.08 & 7.23\\
             & $10^{9}$ & 7.15 & 7.26 & 7.30 &  8.39 & 8.28 & 8.22 & 6.77 & 7.11 & 7.25\\

\hline
\end{tabular}}}
\end{center}
\caption{Condition numbers of \eqref{eq:biot} with $\BB_2$ preconditioner (left),
  $\BB_3$ preconditioner (center) and $\BB_4$ preconditioner (right). In all cases
  $\alpha=1$, $\mu=1$, $c=1$. The preconditioners are not $K$, $\lambda$-robust.
}
\label{tab:biot_subopt_Bother}
\end{table}


 \newpage
 \section{Herrmann formulation of linear elasticity}
 \label{herrmann}
 Let us consider a $2d$ domain on which the equations of linear elasticity are written as 
 \[
 \begin{aligned}
 -\bnabla\cdot\left(2\mu\beps(\bu) + pI\right)&=f &\mbox{ in }\Omega,\\
 \nabla\cdot\bu - \lambda^{-1}p &= 0 &\mbox{ in }\Omega,
 \end{aligned}
 \]
 with pure displacement  boundary conditions. 
 In view of constructing locking-free solvers, 
 one is interested in discretizations that are robust with respect to $\lambda$ (and also with respect to $\mu$).  

 The following  preconditioner is robust: 
\begin{equation}\label{eqB:herrmann}
 \mathcal{B}_H=\begin{bmatrix}
 -\bnabla\cdot\left(2\mu\beps(\bu)\right) & 0 \\
0 & \mu^{-1}(I-\Pi_{\mathbb{R}}) + \lambda^{-1}I
 \end{bmatrix}^{-1},
\end{equation}
as we can see in Table~\ref{table:herrmann}.

 \begin{table}[h]
\begin{center}
   \scriptsize{
    \begin{tabular}{c|c|ccc}
\hline      
      \multirow{2}{*}{$\lambda$} & \multirow{2}{*}{$\mu$} & \multicolumn{3}{c}{$h$} \\
      \cline{3-5}
      & & {$2^{-2}$ } & {$2^{-3}$} & {$2^{-4}$} \\
\hline      
\multirow{8}{*}{1} & $10^{-6}$ & 18.13 & 18.19 & 18.20\\
  & $10^{-4}$ & 18.11 & 18.17 & 18.18\\
  & $10^{-2}$ & 16.22 & 16.27 & 16.28\\
  & 1 & 2.19 & 2.19 & 2.19\\
  & $10^{2}$ & 1.01 & 1.01 & 1.01\\
  & $10^{4}$ & 1.00 & 1.00 & 1.00\\
  & $10^{8}$ & 1.00 & 1.00 & 1.00\\
  & $10^{10}$ & 1.00 & 1.00 & 1.00\\
 \hline
\multirow{8}{*}{$10^{2}$} & $10^{-6}$ & 18.13 & 18.19 & 18.20\\
  & $10^{-4}$ & 18.13 & 18.19 & 18.20\\
  & $10^{-2}$ & 18.11 & 18.17 & 18.18\\
  & 1 & 16.22 & 16.27 & 16.28\\
  & $10^{2}$ & 2.19 & 2.19 & 2.19\\
  & $10^{4}$ & 1.01 & 1.01 & 1.01\\
  & $10^{8}$ & 1.00 & 1.00 & 1.00\\
  & $10^{10}$ & 1.00 & 1.00 & 1.00\\
\hline
    \end{tabular}
    \hspace{10pt}
    \begin{tabular}{c|c|ccc}
\hline      
      \multirow{2}{*}{$\lambda$} & \multirow{2}{*}{$\mu$} & \multicolumn{3}{c}{$h$} \\
      \cline{3-5}
      & & {$2^{-2}$ } & {$2^{-3}$} & {$2^{-4}$} \\
\hline      
\multirow{8}{*}{$10^{4}$} & $10^{-6}$ & 18.13 & 18.19 & 18.20\\
  & $10^{-4}$ & 18.13 & 18.19 & 18.20\\
  & $10^{-2}$ & 18.13 & 18.19 & 18.20\\
  & 1 & 18.11 & 18.17 & 18.18\\
  & $10^{2}$ & 16.22 & 16.27 & 16.28\\
  & $10^{4}$ & 2.19 & 2.19 & 2.19\\
  & $10^{8}$ & 1.00 & 1.00 & 1.00\\
  & $10^{10}$ & 1.00 & 1.00 & 1.00\\
\hline
\multirow{8}{*}{$10^{8}$} & $10^{-6}$ & 18.13 & 18.19 & 18.20\\
  & $10^{-4}$ & 18.13 & 18.19 & 18.20\\
  & $10^{-2}$ & 18.13 & 18.19 & 18.20\\
  & 1 & 18.13 & 18.19 & 18.20\\
  & $10^{2}$ & 18.13 & 18.19 & 18.20\\
  & $10^{4}$ & 18.11 & 18.17 & 18.18\\
  & $10^{8}$ & 2.19 & 2.19 & 2.19\\
  & $10^{10}$ & 1.01 & 1.01 & 1.01\\
\hline
    \end{tabular}
   }
\end{center}
\caption{Condition numbers associated with the preconditioner \eqref{eqB:herrmann} for the Herrmann problem.}\label{table:herrmann}
\end{table}

In regards to Theorem \ref{thm: Braess}, it is clear that for $\lambda=\infty$, the Brezzi conditions
 are satisfied as the problem is then reduced to the 
 Stokes problem on $\mu^{1/2}\bH^1_0(\Omega)\times \mu^{-1/2} L^2_0(\Omega)$. Furthermore, as the bilinear form $a(\cdot, \cdot)$ 
 is coercive in the whole space $\mu^{1/2}\bH^1_0(\Omega)$ the Braess condition is automatically satisfied.

\end{document}